\documentclass{amsart}%
\usepackage{amsfonts}
\usepackage{amsmath}
\usepackage{amssymb}
\usepackage{graphicx}
\usepackage[latin1]{inputenc}
\usepackage[english]{babel}%
\newtheorem{theorem}{Theorem}[section]
\theoremstyle{plain}
\newtheorem{acknowledgement}[theorem]{Acknowledgement}

\newtheorem{definition}[theorem]{Definition}
\newtheorem{example}[theorem]{Example}

\newtheorem{lemma}[theorem]{Lemma}

\newtheorem{proposition}[theorem]{Proposition}
\newtheorem{remark}[theorem]{Remark}

\numberwithin{equation}{section}
\begin{document}
\title[The Non-Archimedean Stochastic Heat Equation]{The Non-Archimedean Stochastic Heat Equation driven by Gaussian Noise}
\author{W. A. Zúñiga-Galindo}
\address{Centro de Investigación y de Estudios Avanzados del Instituto Politécnico Nacional\\
Departamento de Matemáticas, Unidad Querétaro\\
Libramiento Norponiente \#2000, Fracc. Real de Juriquilla. Santiago de
Querétaro, Qro. 76230\\
México.}
\email{wazuniga@math.cinvestav.edu.mx}
\thanks{The author was partially supported by Conacyt (Mexico), Grant \# 127794.}

\begin{abstract}
We introduce and study a new class of non-Archimedean stochastic
pseudodifferential equations. These equations are the non-Archimedean
counterparts of the classical stochastic heat equations. We show the existence
and uniqueness of mild random field solutions for these equations.

\end{abstract}
\maketitle

\section{Introduction}

In this article we study a new class of stochastic pseudodifferential
equations in $\mathbb{R}_{+}\times\mathbb{Q}_{p}^{N}$, here $\mathbb{Q}_{p}$
denotes the field of $p$-adic numbers, driven by a spatially homogeneous
Gaussian noise. More precisely, we consider pseudodifferential equations of
the type%
\[
Lu\left(  t,x\right)  =\sigma\left(  u\left(  t,x\right)  \right)
\overset{\cdot}{W}\left(  t,x\right)  +b\left(  u(t,x\right)  =0\text{, }%
t\geq0\text{, }x\in\mathbb{Q}_{p}^{N}\text{,}%
\]
where $L=$ $\frac{\partial}{\partial t}+A\left(  \partial,\beta\right)  $,
$\beta>0$, with $A\left(  \partial,\beta\right)  $ a pseudodifferential
operator of the form $\mathcal{F}_{x\rightarrow\xi}\left(  A\left(
\partial,\beta\right)  \varphi\right)  =\left\vert a\left(  \xi\right)
\right\vert _{p}^{\beta}\mathcal{F}_{x\rightarrow\xi}\left(  A\varphi\right)
$, and $a\left(  \xi\right)  $ an elliptic polynomial. The coefficients
$\sigma$\ and $b$ are real-valued functions and $\overset{\cdot}{W}\left(
t,x\right)  $ is the formal notation for a Gaussian random perturbation
defined on some probability space. We assume that it is white in time and with
a homogeneous spatial correlation given by a function $f$, see Section
\ref{SectionNoise}. Our main result, see Theorem \ref{Thm2}, asserts the
existence and uniqueness of mild random field solutions for these equations.
The equations studied here are the non-Archimedean counterparts of the
Archimedean stochastic heat equations studied for instance in \cite{Dalang},
\cite{Da Prato} and \cite{Walsh}.

The pseudodifferential equations of the form
\[
\frac{\partial u\left(  t,x\right)  }{\partial t}+A\left(  \partial
,\beta\right)  u\left(  t,x\right)  =0
\]
are the $p$-adic counterparts of the Archimedean heat equations. Indeed, the
fundamental solutions of these equations (i.e. the heat kernels) are
transition density functions of Markov processes on $\mathbb{Q}_{p}^{N}$, see
Section \ref{SectionHeatEquations}. The one-dimensional $p$-adic heat equation
was introduced in \cite[Section \ XVI]{V-V-Z}, since then the theory of such
equations has been steadily developing, see e.g. \cite{A-K-S}, \cite{Koch},
\cite{Ch-Zuniga-2}, \cite{R-Zu}, \cite{Var}, \cite{Zu1} and the references
therein. This type of equations appear in some new models of complex systems
constructed by Avetisov et al., \cite{Av-3}-\cite{Av-5}, thus, the study of
stochastic versions of these equations is a natural and relevant problem.

From a more general perspective, the stochastic processes over the $p$-adics,
or more generally over ultrametric spaces, have attracted a lot of attention
during the last thirty years, see e.g. \cite{A-K1}-\cite{A-K2}, \cite{Av-3}%
-\cite{Av-5}, \cite{Bikulov}, \cite{Bi-Volovich}, \cite{Ch-Zuniga}%
-\cite{Ch-Zuniga-2}, \cite{Evans0}-\cite{Evans}, \cite{Ka-Koch}, \cite{Ka},
\cite{Kamizono}, \cite{Kamizono2}, \cite{Koch1}-\cite{Koch}, \cite{K-Kos}%
-\cite{K-Kos2},\ \cite{T-Z}, \cite{Var}, \cite{V-V-Z}, \cite{Zu1}, and the
references therein. From the point of view of the mathematical physics, the
interest on this type of stochastic processes comes from their connections
with models of complex systems. It has been proposed that the space of states
of certain complex systems, for instance proteins, have a hierarchical
structure, see e.g. \cite{F.et.al}, which can be put in turn in connection
with $p$-adic structures \cite{Av-3}-\cite{Av-5}.

Stochastic equations over $p$-adics have been studied intensively by many
authors, see e.g. \cite{Bikulov}, \cite{Bi-Volovich}, \cite{Evans},
\cite{Ka-Koch}, \cite{Kamizono}, \cite{Kamizono2}, \cite{Koch},\ \cite{K-Kos1}%
, \cite{K-Zy}. The $p$-adic Gaussian noise and the corresponding stochastic
integrals was studied in \cite{Bi-Volovich}, \cite{Evans}, \cite{Kamizono},
\cite{Kamizono2}, \cite{K-Zy}. All these articles consider processes and
stochastic integrals depending on $p$-adic variables. Here, we introduced a
non-Archimedean, spatially homogeneous Gaussian noise parametrized by a
non-negative real variable, the time variable, and by a $p$-adic vector, the
position variable. As far as we know such noises have been not studied before.
On the other hand, in \cite{Koch1}-\cite{Koch} Kochubei introduced stochastic
integrals with respect to the `$p$-adic Brownian motion' generated by the
one-dimensional heat equation. This is a non-Gaussian process parametrized by
non-negative real variable and by a $p$-adic variable.

The article is organized as follows. In Section \ref{Sect2}, we review some
basic facts about $p$-adic analysis. In Section \ref{Sect3}, we review some
aspects of the parabolic type pseudodifferential equations needed for other
sections. In Section \ref{Sect4}, we prove a $p$-adic version of the
Bochner-Schwartz Theorem, see Theorem \ref{Thm1C}. In Section \ref{Sect5}, we
review the stochastic integration with respect to Hilbert-space-valued Wiener
processes, and introduce the Gaussian noise $W$ and its associated cylindrical
process, see Proposition \ref{Prop1}. We also give some results about the
spectral measure of $W$, see Theorem \ref{Thm1}. Finally, we give a result,
Proposition \ref{Prop2}, which gives us \ examples of random distributions
that can be integrated with respect to $W$. It is interesting to note that the
proof of Proposition \ref{Prop2} is much more involved than the corresponding
result in the Archimedean setting, see e.g. proof of Proposition 3.3 in
\cite{Dalang-Quer}. This is due to the fact that in the $p$-adic setting, the
smoothing of a process requires `cutting' and convolution operations while, in
the Archimedean setting, it requires only a convolution operation. In Section
\ref{Sect6}, we prove the main result, see Theorem \ref{Thm2}. Like in
\cite{Dalang} we prove Theorem \ref{Thm2} under the `Hypotheses A and B,' here
we give \ an explicit and sufficient conditions to fulfill these hypotheses in
terms of the spectral measure of $W$, see Theorem \ref{Thm1} and Lemma
\ref{lemma5A}.

\section{\label{Sect2}$p$-adic analysis: essential ideas}

In this section we fix the notation and collect some basic results on $p$-adic
analysis that we will use through the article. For a detailed exposition the
reader may consult \cite{A-K-S}, \cite{Taibleson}, \cite{V-V-Z}.

\subsection{The field of $p$-adic numbers}

Along this article $p$ will denote a prime number. The field of $p-$adic
numbers $\mathbb{Q}_{p}$ is defined as the completion of the field of rational
numbers $\mathbb{Q}$ with respect to the $p-$adic norm $|\cdot|_{p}$, which is
defined as
\[
|x|_{p}=%
\begin{cases}
0 & \text{if }x=0\\
p^{-\gamma} & \text{if }x=p^{\gamma}\dfrac{a}{b},
\end{cases}
\]
where $a$ and $b$ are integers coprime with $p$. The integer $\gamma:=ord(x)$,
with $ord(0):=+\infty$, is called the\textit{ }$p-$\textit{adic order of} $x$.
We extend the $p-$adic norm to $\mathbb{Q}_{p}^{N}$ by taking%
\[
||x||_{p}:=\max_{1\leq i\leq N}|x_{i}|_{p},\qquad\text{for }x=(x_{1}%
,\dots,x_{N})\in\mathbb{Q}_{p}^{N}.
\]

We define $ord(x)=\min_{1\leq i\leq N}\{ord(x_{i})\}$, then $||x||_{p}%
=p^{-\text{ord}(x)}$. The metric space $\left(  \mathbb{Q}_{p}^{N}%
,||\cdot||_{p}\right)  $ is a complete ultrametric space. As a topological
space $\mathbb{Q}_{p}$\ is homeomorphic to a Cantor-like subset of the real
line, see e.g. \cite{A-K-S}, \cite{V-V-Z}.

Any $p-$adic number $x\neq0$ has a unique expansion $x=p^{ord(x)}\sum
_{j=0}^{+\infty}x_{j}p^{j}$, where $x_{j}\in\{0,1,2,\dots,p-1\}$ and
$x_{0}\neq0$. By using this expansion, we define \textit{the fractional part
of }$x\in\mathbb{Q}_{p}$, denoted $\{x\}_{p}$, as the rational number
\[
\{x\}_{p}=%
\begin{cases}
0 & \text{if }x=0\text{ or }ord(x)\geq0\\
p^{\text{ord}(x)}\sum_{j=0}^{-ord(x)-1}x_{j}p^{j} & \text{if }ord(x)<0.
\end{cases}
\]
For $\gamma\in\mathbb{Z}$, denote by $B_{\gamma}^{N}(a)=\{x\in\mathbb{Q}%
_{p}^{N}:||x-a||_{p}\leq p^{\gamma}\}$ \textit{the ball of radius }$p^{\gamma
}$ \textit{with center at} $a=(a_{1},\dots,a_{N})\in\mathbb{Q}_{p}^{N}$, and
take $B_{\gamma}^{N}(0):=B_{\gamma}^{N}$. Note that $B_{\gamma}^{N}%
(a)=B_{\gamma}(a_{1})\times\cdots\times B_{\gamma}(a_{n})$, where $B_{\gamma
}(a_{i}):=\{x\in\mathbb{Q}_{p}:|x-a_{i}|_{p}\leq p^{\gamma}\}$ is the
one-dimensional ball of radius $p^{\gamma}$ with center at $a_{i}\in
\mathbb{Q}_{p}$. The ball $B_{0}^{N}$ equals the product of $N$ copies of
$B_{0}:=\mathbb{Z}_{p}$, \textit{the ring of }$p-$\textit{adic integers}. We
denote by $\Omega(\left\Vert x\right\Vert _{p})$ the characteristic function
of $B_{0}^{N}$. For more general sets, say Borel sets, we use ${\LARGE 1}%
_{A}\left(  x\right)  $ to denote the characteristic function of $A$.

\subsection{The Bruhat-Schwartz space}

A complex-valued function $\varphi$ defined on $\mathbb{Q}_{p}^{N}$ is
\textit{called locally constant} if for any $x\in\mathbb{Q}_{p}^{N}$ there
exists an integer $l(x)\in\mathbb{Z}$ such that%
\begin{equation}
\varphi(x+x^{\prime})=\varphi(x)\text{ for }x^{\prime}\in B_{l(x)}^{N}.
\label{local_constancy}%
\end{equation}
A function $\varphi:\mathbb{Q}_{p}^{N}\rightarrow\mathbb{C}$ is called a
\textit{Bruhat-Schwartz function (or a test function)} if it is locally
constant with compact support. The $\mathbb{C}$-vector space of
Bruhat-Schwartz functions is denoted by $\mathcal{D}(\mathbb{Q}_{p}%
^{N}):=\mathcal{D}$. We will denote by $\mathcal{D}_{\mathbb{R}}%
(\mathbb{Q}_{p}^{N})$\ the $\mathbb{R}$-vector space of Bruhat-Schwartz functions.

For $\varphi\in\mathcal{D}(\mathbb{Q}_{p}^{N})$, the largest of such number
$l=l(\varphi)$ satisfying (\ref{local_constancy}) is called \textit{the
exponent of local constancy of} $\varphi$.

Let $\mathcal{D}^{\prime}(\mathbb{Q}_{p}^{N}):=\mathcal{D}^{\prime}$ denote
the set of all functionals (distributions) on $\mathcal{D}(\mathbb{Q}_{p}%
^{N})$. All functionals on $\mathcal{D}(\mathbb{Q}_{p}^{N})$ are continuous.

Set $\chi_{p}(y)=\exp(2\pi i\{y\}_{p})$ for $y\in\mathbb{Q}_{p}$. The map
$\chi_{p}(\cdot)$ is an additive character on $\mathbb{Q}_{p}$, i.e. a
continuos map from $\mathbb{Q}_{p}$ into the unit circle satisfying $\chi
_{p}(y_{0}+y_{1})=\chi_{p}(y_{0})\chi_{p}(y_{1})$, $y_{0},y_{1}\in
\mathbb{Q}_{p}$.

Given $\xi=(\xi_{1},\dots,\xi_{N})$ and $x=(x_{1},\dots,x_{N})\in
\mathbb{Q}_{p}^{N}$, we set $\xi\cdot x:=\sum_{j=1}^{N}\xi_{j}x_{j}$. The
Fourier transform of $\varphi\in\mathcal{D}(\mathbb{Q}_{p}^{N})$ is defined
as
\[
(\mathcal{F}\varphi)(\xi)=\int_{\mathbb{Q}_{p}^{N}}\chi_{p}(-\xi\cdot
x)\varphi(\xi)d^{N}x\quad\text{for }\xi\in\mathbb{Q}_{p}^{N},
\]
where $d^{N}x$ is the Haar measure on $\mathbb{Q}_{p}^{N}$ normalized by the
condition $vol(B_{0}^{N})=1$. The Fourier transform is a linear isomorphism
from $\mathcal{D}(\mathbb{Q}_{p}^{N})$ onto itself satisfying $(\mathcal{F}%
(\mathcal{F}\varphi))(\xi)=\varphi(-\xi)$. We will also use the notation
$\mathcal{F}_{x\rightarrow\xi}\varphi$ and $\widehat{\varphi}$\ for the
Fourier transform of $\varphi$.

\subsubsection{Fourier transform}

The Fourier transform $\mathcal{F}\left[  T\right]  $ of a distribution
$T\in\mathcal{D}^{\prime}\left(  \mathbb{Q}_{p}^{N}\right)  $ is defined by%
\[
\left(  \mathcal{F}\left[  T\right]  ,\varphi\right)  =\left(  T,\mathcal{F}%
\left[  \varphi\right]  \right)  \text{ for all }\varphi\in\mathcal{D}\left(
\mathbb{Q}_{p}^{N}\right)  \text{.}%
\]
The Fourier transform $f\rightarrow\mathcal{F}\left[  T\right]  $ is a linear
isomorphism from $\mathcal{D}^{\prime}\left(  \mathbb{Q}_{p}^{N}\right)
$\ onto $\mathcal{D}^{\prime}\left(  \mathbb{Q}_{p}^{N}\right)  $.
Furthermore, $T=\mathcal{F}\left[  \mathcal{F}\left[  T\right]  \left(
-\xi\right)  \right]  $.

\section{\label{Sect3}$p$-adic parabolic type pseudodifferential equations}

In this article we work exclusively with complex and real valued functions on
$\mathbb{Q}_{p}^{N}$. Having complex and real valued functions defined on a
locally compact topological group, we have the notion of continuous function
and may use the functional spaces $L^{\varrho}(\mathbb{Q}_{p}^{N},d^{N}x)$,
$\rho\geq1$ defined in the standard way. \ We also use the following standard notation:

\begin{itemize}
\item[(i)] $C(I, X)$ the space of continuous functions $u$ on a time interval
$I$ with values in $X$;

\item[(ii)] $C^{1}(I, X)$ the space of continuously differentiable functions
$u$ on a time interval $I$ such that $u^{\prime}\in X$;

\item[(iii)] $L^{1}(I, X)$ the space of measurable functions $u$ on $I$ with
values in $X$ such that $\| u\|$ is integrable;

\item[(iv)] $W^{1,1}(I,X)$ the space of measurable functions $u$ on $I$ with
values in $X$ such that $u^{\prime}\in L^{1}(I,X)$.
\end{itemize}

\subsection{Elliptic pseudodifferential operators}

Let $a\left(  \xi\right)  \in\mathbb{Q}_{p}\left[  \xi_{1},\ldots,\xi
_{n}\right]  $ be a non-constant polynomial. We say that $a\left(  \xi\right)
$ is \textit{an elliptic polynomial of degree }$d$, if it satisfies: (i)
$a\left(  \xi\right)  $ is a homogeneous polynomial of degree $d$, and (ii)
$a\left(  \xi\right)  =0$ $\Leftrightarrow$ $\xi=0$.

It is known that an elliptic polynomial satisfies%
\begin{equation}
C_{0}\left\Vert \xi\right\Vert _{p}^{d}\leq\left\vert a\left(  \xi\right)
\right\vert _{p}\leq C_{1}\left\Vert \xi\right\Vert _{p}^{d}\text{, for every
}\xi\in\mathbb{Q}_{p}^{n}, \label{elliptic}%
\end{equation}
for some positive constants $C_{0}=C_{0}(a)$, $C_{1}=C_{1}(a)$, cf.
\cite[Lemma 1]{Zu1}. Without loss of generality we will assume that $a\left(
\xi\right)  \in\mathbb{Z}_{p}\left[  \xi_{1},\ldots,\xi_{n}\right]  $.

Given a fixed $\beta>0$, a pseudodifferential operator of the form
$\boldsymbol{A}\left(  \partial,\beta\right)  \phi\left(  x\right)
=\mathcal{F}_{\xi\rightarrow x}^{-1}\left(  \left\vert a\left(  \xi\right)
\right\vert _{p}^{\beta}\mathcal{F}_{x\rightarrow\xi}\phi\right)  $, $\phi
\in\mathcal{D}$, is called \textit{an elliptic pseudodifferential operator of
degree d with symbol} $\left\vert a\right\vert _{p}^{\beta}$.

\begin{lemma}
\label{Lemmasemigroup}With the above notation the following assertions hold:

\noindent(i)%
\[%
\begin{array}
[c]{ccc}%
\mathcal{D} & \rightarrow & C\left(  \mathbb{Q}_{p}^{N},\mathbb{C}\right)
\cap L^{2}\left(  \mathbb{Q}_{p}^{N},d^{N}x\right) \\
&  & \\
\phi & \rightarrow & \boldsymbol{A}\left(  \partial,\beta\right)  \phi;
\end{array}
\]
\noindent(ii) the closure of the operator $\boldsymbol{A}\left(
\partial,\beta\right)  $, $\beta>0$ (let us denote it by $\boldsymbol{A}%
\left(  \partial,\beta\right)  $ again) with domain
\begin{equation}
Dom\left(  \boldsymbol{A}\left(  \partial,\beta\right)  \right)
:=Dom(\boldsymbol{A})=\left\{  f\in L^{2}:\left\vert a\right\vert _{p}^{\beta
}\widehat{f}\in L^{2}\right\}  \label{DomDa}%
\end{equation}
is a self-adjoint operator;

\noindent(iii) $-\boldsymbol{A}\left(  \partial,\beta\right)  $ is the
infinitesimal generator of a contraction $C_{0}$ semigroup $\left(
\mathcal{T}(t)\right)  _{t\geq0}$;

\noindent(iv) set
\begin{equation}
\Gamma(t,x):=\mathcal{F}_{\xi\rightarrow x}^{-1}\left(  e^{-t\left\vert
a\left(  \xi\right)  \right\vert _{p}^{\beta}}\right)  \text{ for }%
t>0,x\in\mathbb{Q}_{p}^{N}. \label{heatkernel}%
\end{equation}
Then%
\[
\left(  \mathcal{T}(t)f\right)  \left(  x\right)  =\left\{
\begin{array}
[c]{lll}%
\left(  \Gamma(t,\cdot)\ast f\right)  \left(  x\right)  & \text{for} & t>0\\
&  & \\
f\left(  x\right)  & \text{for} & t=0,
\end{array}
\right.
\]
for $f\in L^{2}$.
\end{lemma}

\begin{proof}
The results follow from the properties of the heat kernels given in \cite{Zu1}
by using well-known techniques of semigroup theory, see e.g. \cite{C-H}.
Alternatively, the reader may consult \ \cite[Lemma 3.21, Lemma 3.23, Lemma
7.4, Theorem 7.5]{T-Z}\ for same results in a more general setting.
\end{proof}

\subsection{\label{SectionHeatEquations}$p$-adic heat equations}

Consider the following Cauchy problem:%

\begin{equation}
\left\{
\begin{array}
[c]{lll}%
\frac{\partial u\left(  t,x\right)  }{\partial t}+\boldsymbol{A}\left(
\partial,\beta\right)  u\left(  t,x\right)  =f\left(  t,x\right)  , &
x\in\mathbb{Q}_{p}^{N}, & t\in\left[  0,T\right] \\
&  & \\
u\left(  0,x\right)  =u_{0}\left(  x\right)  \in Dom(\boldsymbol{A}). &  &
\end{array}
\right.  \label{CauchyProblem}%
\end{equation}
We say that a function $u(x,t)$ is a \textit{solution of} (\ref{CauchyProblem}%
) if \linebreak$u\in C\left(  [0,T],Dom(\boldsymbol{A})\right)  \cap
C^{1}\left(  [0,T],L^{2}\right)  $ and $u$ satisfies equation
(\ref{CauchyProblem}) for all $t\in\left[  0,T\right]  $.

\begin{theorem}
\label{Thm1A}Let $\beta>0$ and let $f\in C\left(  [0,T],L^{2}\right)  $.
Assume that at least one of the following conditions is satisfied:

\begin{itemize}
\item[(i)] $f\in L^{1}\left(  (0,T),Dom(\boldsymbol{A})\right)  $;

\item[(ii)] $f\in W^{1,1}((0,T),L^{2})$.
\end{itemize}

Then Cauchy problem (\ref{CauchyProblem}) has a unique solution given by
\[
u(t,x)=\int_{\mathbb{Q}_{p}^{N}}\Gamma\left(  t,x-y\right)  u_{0}\left(
y\right)  d^{N}y+\int_{0}^{t}\left\{  \int_{\mathbb{Q}_{p}^{N}}\Gamma\left(
t-\tau,x-y\right)  f\left(  \tau,y\right)  d^{n}y\right\}  d\tau,
\]
where $\Gamma$ is defined in (\ref{heatkernel}).
\end{theorem}

\begin{proof}
The \ result follows from Lemma \ref{Lemmasemigroup} by well-known results in
semigroup theory, see e.g. \cite{C-H}. Alternatively, the reader may consult
\cite[ Theorem 7.9]{T-Z}\ for same result in a more general setting.
\end{proof}

\begin{theorem}
\label{Thm1B} The heat kernel (or fundamental solution of (\ref{CauchyProblem}%
)) $\Gamma\left(  t,x\right)  $, $t>0$, satisfies the following:

\noindent(i) $\Gamma\left(  t,x\right)  \geq0$ for any $t>0$;

\noindent(ii) $\int_{\mathbb{Q}_{p}^{N}}\Gamma\left(  t,x\right)  d^{N}x=1$
for any $t>0$;

\noindent(iii) $\Gamma\left(  t,\cdot\right)  \in L^{1}(\mathbb{Q}_{p}^{N})$
for any $t>0$;

\noindent(iv) $\left(  \Gamma\left(  t,\cdot\right)  \ast\Gamma\left(
t^{\prime},\cdot\right)  \right)  \left(  x\right)  =\Gamma\left(
t+t^{\prime},x\right)  $ for any $t$, $t^{\prime}>0$;

\noindent(v) $\lim_{t\rightarrow0+}\Gamma\left(  t,x\right)  =\delta\left(
x\right)  $ in $\mathcal{D}^{\prime}$;

\noindent(vi) $\Gamma\left(  t,x\right)  \leq At\left(  t^{\frac{1}{d\beta}%
}+\left\Vert x\right\Vert _{p}\right)  ^{-d\beta-N}$ for any $x$
$\in\mathbb{Q}_{p}^{N}$ and $t>0$;

\noindent(vii) $\Gamma(x,t)$ is the transition density of a time- and space
homogenous Markov process which is bounded, right-continuous and has no
discontinuities other than jumps.
\end{theorem}

\begin{proof}
See Theorem 1, Theorem 2, Proposition 2 and Theorem 4 in \cite{Zu1}.
\end{proof}

The $p$-adic heat equation in dimension one was introduced in the book of
Vladimirov, Volovich and Zelenov \cite[Section \ XVI]{V-V-Z}. In
\cite[Chapters 4, 5]{Koch} Kochubei presented a general theory for
one-dimensional parabolic-type pseudodifferential equations with variable
coefficients, whose fundamental solutions are transition density functions for
Markov processes in the $p$-adic line. For a generalization of this theory see
\cite{Ch-Zuniga-2}. In \cite{Zu1} the author introduced the elliptic operators
presented before and studied the corresponding $n$-dimensional heat equations
and the associated Markov processes.

\section{\label{Sect4}Positive-definite distributions and the Bochner-Schwartz
theorem}

In this section, we establish a $p$-adic version of the\ Bochner-Schwartz
Theorem on positive-definite distributions following to Gel'fand and Vilenkin
\cite[Chapter II]{Gel-Vil}. 

\subsection{The $p$-adic Bochner-Schwartz theorem}

Along this section we work with complex-valued test functions. A distribution
$F\in\mathcal{D}^{\prime}\left(  \mathbb{Q}_{p}^{N}\right)  $ is called
\textit{positive}, if $\left(  F,\varphi\right)  \geq0$ for every positive
test function $\varphi$, i.e. if $\varphi\left(  x\right)  \geq0$ for every
$x$. In this case we will use the notation $F\geq0$. We say that $F$ is
\textit{multiplicatively positive}, if $\left(  F,\varphi\overline{\varphi
}\right)  \geq0$ for every test function $\varphi$, where $\overline{\varphi}$
denotes the complex conjugate of $\varphi$. A distribution $F$ is
positive-definite, if for every test function $\varphi$, the inequality
$\left(  F,\overline{\varphi\ast\widetilde{\varphi}}\right)  \geq0$ holds,
where $\widetilde{\varphi}\left(  x\right)  =\overline{\varphi\left(
-x\right)  }$.

\begin{theorem}
[$p$-adic Bochner-Schwartz Theorem]\label{Thm1C}Every positive-definite
distribution $F$ on $\mathbb{Q}_{p}^{N}$ is the Fourier transform of a regular
Borel measure $\mu$ on $\mathbb{Q}_{p}^{N}$, i.e.
\[
\left(  F,\varphi\right)  =%
{\displaystyle\int\nolimits_{\mathbb{Q}_{p}^{N}}}
\widehat{\varphi}\left(  \xi\right)  d\mu\left(  \xi\right)  \text{ for
}\varphi\in\mathcal{D}\left(  \mathbb{Q}_{p}^{N}\right)  .
\]
Conversely, the Fourier transform of any regular Borel measure gives rise to a
positive-definite distribution on $\mathbb{Q}_{p}^{N}$.
\end{theorem}

\begin{proof}
($\Rightarrow$) By the Riesz-Markov-Kakutani Theorem every positive
distribution $F$ on $\mathbb{Q}_{p}^{N}$ has the form%
\[
\left(  F,\phi\right)  =%
{\displaystyle\int\nolimits_{\mathbb{Q}_{p}^{N}}}
\phi\left(  \xi\right)  d\mu\left(  \xi\right)  \text{ for }\phi\in
\mathcal{D}\left(  \mathbb{Q}_{p}^{N}\right)  ,
\]
where $\mu$ is a regular Borel measure. Conversely, \ every regular Borel
measure $\mu$ defines a positive linear functional on $\mathcal{D}\left(
\mathbb{Q}_{p}^{N}\right)  $. On the other hand, since $F$ is a
multiplicatively positive distribution if and only if $F$ is a positive
distribution, we can replace positive by multiplicatively positive in the
above assertion. We now note that the Fourier transform carries
positive-definite distributions into multiplicatively positive distributions,
and every multiplicatively positive distribution can be obtained in this
manner. Indeed,%
\begin{align*}
\left(  \widehat{F},\widehat{\varphi}\overline{\widehat{\varphi}}\right)   &
=\left(  \widehat{F},\widehat{\varphi\ast\widetilde{\varphi}}\right) \\
&  =\left(  F,\widehat{\widehat{\varphi\ast\widetilde{\varphi}}}\right)
=\left(  F\left(  \xi\right)  ,\left(  \varphi\ast\widetilde{\varphi}\right)
\left(  -\xi\right)  \right) \\
&  =\left(  F,\overline{\varphi\ast\widetilde{\varphi}}\right)  ,
\end{align*}
since $\overline{\left(  \varphi\ast\widetilde{\varphi}\right)  \left(
-\xi\right)  }$ $=\left(  \varphi\ast\widetilde{\varphi}\right)  \left(
\xi\right)  $. Now, let $F\in\mathcal{D}^{\prime}\left(  \mathbb{Q}_{p}%
^{N}\right)  $ be a multiplicatively positive distribution, i.e. $\left(
F,\psi\overline{\psi}\right)  \geq0$ for every $\psi\in\mathcal{D}\left(
\mathbb{Q}_{p}^{N}\right)  $. Then, there exist a distribution $T$ and a test
function $\phi$ satisfying $\widehat{T}=F$ and $\psi=\widehat{\phi}$, because
the Fourier transform is an isomorphism on $\mathcal{D}^{\prime}\left(
\mathbb{Q}_{p}^{N}\right)  $ and on $\mathcal{D}\left(  \mathbb{Q}_{p}%
^{N}\right)  $. From this observation we have $\left(  F,\psi\overline{\psi
}\right)  =\left(  T,\overline{\varphi\ast\widetilde{\varphi}}\right)  \geq0$.

($\Leftarrow$) It follows from this calculation:%
\begin{align*}%
{\displaystyle\int\nolimits_{\mathbb{Q}_{p}^{N}}}
\mathcal{F}\left(  \overline{\left(  \varphi\ast\widetilde{\varphi}\right)
}\right)  \left(  \xi\right)  d\mu\left(  \xi\right)   &  =%
{\displaystyle\int\nolimits_{\mathbb{Q}_{p}^{N}}}
\left(  \mathcal{F}^{-1}\varphi\right)  \left(  \xi\right)  \overline{\left(
\mathcal{F}^{-1}\varphi\right)  }\left(  \xi\right)  d\mu\left(  \xi\right) \\
&  =%
{\displaystyle\int\nolimits_{\mathbb{Q}_{p}^{N}}}
\left\vert \left(  \mathcal{F}^{-1}\varphi\right)  \left(  \xi\right)
\right\vert ^{2}d\mu\left(  \xi\right)  \geq0.
\end{align*}

\end{proof}

\subsection{Positive-definite functions}

We recall \ that a continuous function $g:\mathbb{Q}_{p}^{N}\rightarrow
\mathbb{C}$ is positive-definite, if for any $p$-adic numbers $x_{1}%
,\ldots,x_{m}$ and any complex numbers $\sigma_{1},\ldots,\sigma_{m}$, it
verifies that $\sum_{j}\sum_{i}g\left(  x_{j}-x_{i}\right)  \sigma
_{j}\overline{\sigma}_{i}\geq0$. Such function $g$ satisfies the following:
$\overline{g}$ is positive-definite, $g(-x)=\overline{g}(x)$, $g(0)\geq0$, and
$\left\vert g(x)\right\vert \leq g(0)$. We associate to $g$ the distribution
$\int_{\mathbb{Q}_{p}^{N}}g(x)\varphi\left(  x\right)  d^{N}x$, while
Gel'fand-Vilenkin attach to $g$ the distribution $\int_{\mathbb{Q}_{p}^{N}%
}\overline{g}(x)\varphi\left(  x\right)  d^{N}x$, for this reason our
definition of positive-definite distribution is slightly different, but
equivalent to the one given in \cite[Chapter II]{Gel-Vil}.\ Finally, we recall
that $g$ satisfies $\left(  g,\overline{\varphi\ast\widetilde{\varphi}%
}\right)  \geq0$ for any test function $\varphi$, i.e. $g$ generates a
positive-definite distribution, see e.g. \cite[Proposition 4.1]{Berg-Forst}.

\section{\label{Sect5}Stochastic integrals and Gaussian noise}

In this section we introduce the stochastic integration with respect to a
spatially homogeneous Gaussian noise. Our exposition has been strongly
influenced by \cite{Dalang-Quer}. There are two distinct approaches (or
schools) of study for stochastic partial differential equations, based on
different theories of stochastic integration: the Walsh theory \cite{Walsh},
which uses integration with respect to worthy martingale measures, and a
theory of integration with respect to Hilbert-space valued processes \cite{Da
Prato}. In \cite{Dalang-Quer} the authors discuss the connections between
these theories. In this article we use Hilbert-space approach. In this section
we present the non-Archimedean counterpart of this theory.

\subsection{\label{StochasticInt}Stochastic integrals with respect to a
spatially homogeneous Gaussian noise}

Let $V$ be a separable Hilbert space with inner product $\left\langle
\cdot,\cdot\right\rangle _{V}$. Following \cite{Dalang-Quer} and the
references therein, we define the general notion of cylindrical Wiener process
\ in $V$ as follows:

\begin{definition}
\label{Def1}Let $Q$ be a symmetric and non-negative definite bounded linear
operator \ on $V$. A family of random variables $B=\left\{  B_{t}\left(
h\right)  ,t\geq0,h\in V\right\}  $ is a cylindrical Wiener process if the
following conditions hold:

\noindent(i) for any $h\in V$, $\left\{  B_{t}\left(  h\right)  ,t\geq
0\right\}  $ defines a Brownian motion with variance $t\left\langle
Qh,h\right\rangle _{V}$;

\noindent(ii) for all $s$, $t\in\mathbb{R}_{+}$ and $h$, $g\in V$,
\[
E\left(  B_{s}\left(  h\right)  B_{t}\left(  g\right)  \right)  =\left(
s\wedge t\right)  \left\langle Qh,g\right\rangle _{V},
\]
where $s\wedge t:=\min\left\{  s,t\right\}  $. If $Q=I_{V}$ is the identity
operator in $V$, then $B$ will be called a standard cylindrical Wiener
process. We will refer to $Q$ as the covariance of $B$.
\end{definition}

Let $\mathcal{F}_{t}$ be the $\sigma$-field generated by the random variables
$\left\{  B_{s}\left(  h\right)  ,h\in V,0\leq s\leq t\right\}  $ and the
$P$-null sets. We define the \textit{predictable }$\sigma$\textit{-field in
}$\left[  0,T\right]  \times\Omega$ generated by the sets%
\[
\left\{  \left(  s,t\right]  \times A,A\in\mathcal{F}_{s}\text{, }0\leq
s<t\leq T\right\}  .
\]

We denote by $V_{Q}$ the completion of the Hilbert space $V$ endowed with the
inner semi-product%
\[
\left\langle h,g\right\rangle _{V_{Q}}:=\left\langle Qh,g\right\rangle
_{V}\text{, \ }h\text{,}g\in V.
\]

We define the stochastic integral of any predictable square-integrable process
with values in $V_{Q}$ as follows. Let $\left(  v_{j}\right)  _{j}$ be a
complete orthonormal basis of $V_{Q}$. For any predictable process $g\in
L^{2}\left(  \Omega\times\left[  0,T\right]  ;V_{Q}\right)  $, the following
series converges in $L^{2}\left(  \Omega,\mathcal{F},P\right)  $ and the sum
does not depend on \ the chosen orthonormal basis:%
\begin{equation}
g\cdot B:=%
{\displaystyle\sum\limits_{j=1}^{+\infty}}
{\displaystyle\int\nolimits_{0}^{T}}
\left\langle g_{s},v_{j}\right\rangle _{V_{Q}}dB_{s}\left(  v_{j}\right)  .
\label{Eq1}%
\end{equation}
We note that each summand in the above series is a classical Itô integral with
respect to a standard Brownian motion, and the resulting stochastic integral
is a real-valued random variable. The stochastic integral $g\cdot B$ is also
denoted by $\int_{0}^{T}g_{s}dB_{s}$. The independence of each of terms in
series (\ref{Eq1}) leads to the isometry property%
\[
E\left(  \left(  g\cdot B\right)  ^{2}\right)  =E\left(  \left(  \int_{0}%
^{T}g_{s}dB_{s}\right)  ^{2}\right)  =E\left(  \left(  \int_{0}^{T}\left\Vert
g_{s}\right\Vert _{V_{Q}}ds\right)  \right)  .
\]

\subsection{\label{SectionNoise}Spatially homogeneous Gaussian noise}

Let $\left(  \Omega,\mathcal{F},P\right)  $ be a complete probability space.
We denote by $\mathcal{I}(\mathbb{R})$ the $\mathbb{R}$-vector space of
functions of the form $\sum_{k=1}^{m}c_{k}1_{I_{k}}\left(  x\right)  $ where
$c_{1},\ldots,c_{k}$ are real numbers and each $I_{k}$ is a bounded interval
(open, closed, half-open). It is well-known\ that $\mathcal{I}(\mathbb{R})$ is
dense in $L^{\rho}(\mathbb{R})$ for $1\leq\rho<\infty$.

We denote by $\mathcal{I}(\mathbb{R})\otimes_{\text{alg}}\mathcal{D}%
_{\mathbb{R}}(\mathbb{Q}_{p}^{N})$ the algebraic tensor product of the
$\mathbb{R}$-vector spaces $\mathcal{I}(\mathbb{R})$ and $\mathcal{D}%
_{\mathbb{R}}(\mathbb{Q}_{p}^{N})$. Notice that $\mathcal{I}(\mathbb{R}%
)\otimes_{\text{alg}}\mathcal{D}_{\mathbb{R}}(\mathbb{Q}_{p}^{N})$ is the
$\mathbb{R}$-vector space spanned by $\sum_{j\in J}c_{j}\left(  t\right)
\Omega\left(  p^{m}\left\Vert x-\widetilde{x}_{j}\right\Vert _{p}\right)  $,
where $c_{j}\left(  t\right)  \in\mathcal{I}(\mathbb{R})$, $m\in\mathbb{Z}$,
and $J$ is a finite set.

\ On $\left(  \Omega,\mathcal{F},P\right)  $, we consider a family of mean
zero Gaussian random variables
\begin{equation}
\left\{  W\left(  \varphi\right)  ,\varphi\in\mathcal{I}(\mathbb{R}%
)\otimes_{\text{alg}}\mathcal{D}_{\mathbb{R}}(\mathbb{Q}_{p}^{N})\right\}
\label{Eq3}%
\end{equation}
with covariance%
\begin{align}
E\left(  W\left(  \varphi\right)  W\left(  \psi\right)  \right)   &  =%
{\displaystyle\int\nolimits_{0}^{+\infty}}
{\displaystyle\int\nolimits_{\mathbb{Q}_{p}^{N}}}
{\displaystyle\int\nolimits_{\mathbb{Q}_{p}^{N}}}
\varphi(t,x)f(x-y)\psi(t,y)d^{N}xd^{N}ydt\nonumber\\
&  =%
{\displaystyle\int\nolimits_{0}^{+\infty}}
{\displaystyle\int\nolimits_{\mathbb{Q}_{p}^{N}}}
f(z)\left(  \varphi(t)\ast\widetilde{\psi}(t)\right)  \left(  z\right)
d^{N}zdt, \label{Eq2}%
\end{align}
where $\widetilde{\psi}(t)\left(  z\right)  =\psi(t,-z)$ and $f$ is
non-negative continuous function on $\mathbb{Q}_{p}^{N}\smallsetminus\left\{
0\right\}  $. This function induces a positive distribution on $\mathbb{Q}%
_{p}^{N}$ and then $f$ is the Fourier transform of a regular Borel measure
$\mu$ on $\mathbb{Q}_{p}^{N}$, see Theorem \ref{Thm1C}. This measure is
\textit{called the spectral measure} of $W$. In this \ case%
\[
E\left(  W\left(  \varphi\right)  W\left(  \psi\right)  \right)  =%
{\displaystyle\int\nolimits_{0}^{+\infty}}
{\displaystyle\int\nolimits_{\mathbb{Q}_{p}^{N}}}
\mathcal{F}\varphi(t)\left(  \xi\right)  \overline{\mathcal{F}\psi(t)\left(
\xi\right)  }d\mu\left(  \xi\right)  dt.
\]

\subsubsection{Some examples of kernels}

The basic example of kernel function is the white noise kernel: $f(x)=\delta
\left(  x\right)  $, $d\mu\left(  \xi\right)  =d^{N}\xi$. Here are some
typical examples:

\begin{example}
If $d\mu\left(  \xi\right)  =\left\Vert \xi\right\Vert _{p}^{-\alpha}d^{N}\xi
$, $0<\alpha<N$, then $f(x)=R_{\alpha}(x)=\frac{1-p^{-\alpha}}{1-p^{\alpha-N}%
}\left\Vert x\right\Vert _{p}^{\alpha-n}$, the Riesz kernel, see e.g.
\cite[Chapter III, Section 4]{Taibleson}.
\end{example}

\begin{example}
If $d\mu\left(  \xi\right)  =e^{-\left\Vert \xi\right\Vert _{p}^{\beta}}$,
$\beta>0$, then $f\left(  x\right)  =\mathcal{F}_{\xi\rightarrow x}\left(
e^{-\left\Vert \xi\right\Vert _{p}^{\beta}}\right)  $ is the $p$-adic heat
kernel. Notice that we can replace $\left\Vert \xi\right\Vert _{p}^{\beta}$ by
$\left\vert a\left(  \xi\right)  \right\vert _{p}^{\beta}$, where $a\left(
\xi\right)  $ is an elliptic polynomial.
\end{example}

Before presenting our next example, we recall the following result:

\begin{lemma}
[{\cite[Lemma 5.2]{Taibleson}}]\label{lemma0}Suppose that $\alpha>0$. Define%
\[
K_{\alpha}\left(  x\right)  =\left\{
\begin{array}
[c]{lll}%
\frac{1-p^{-\alpha}}{1-p^{\alpha-N}}\left(  \left\Vert x\right\Vert
_{p}^{\alpha-N}-p^{\alpha-N}\right)  \Omega\left(  \left\Vert x\right\Vert
_{p}\right)  & \text{if} & \alpha\neq N\\
&  & \\
\left(  1-p^{-N}\right)  \log_{p}\left(  \frac{p}{\left\Vert x\right\Vert
_{p}}\right)  \Omega\left(  \left\Vert x\right\Vert _{p}\right)  & \text{if} &
\alpha=N.
\end{array}
\right.
\]
Then $K_{\alpha}\in L^{1}$ and $\mathcal{F}K_{\alpha}\left(  \xi\right)
=\max\left(  1,\left\Vert \xi\right\Vert _{p}\right)  ^{-\alpha}$.
\end{lemma}

The distribution $K_{\alpha}$ is called\textit{ the Bessel potential of order
}$\alpha$, see e.g. \cite[Chapter III, Section 5]{Taibleson}.

\begin{example}
If $d\mu\left(  \xi\right)  =\max\left(  1,\left\Vert \xi\right\Vert
_{p}\right)  ^{-\alpha}$, $\alpha>0$, then $f(x)=K_{\alpha}\left(  x\right)
$, the Bessel potential of order $\alpha$.
\end{example}

\subsubsection{A Cylindrical Wiener process associated with $W$}

Let $U$ be \ the completion of the Bruhat-Schwartz space $\mathcal{D}%
_{\mathbb{R}}(\mathbb{Q}_{p}^{N})$ endowed with semi-inner product%
\[
\left\langle \varphi,\psi\right\rangle _{U}:=%
{\displaystyle\int\nolimits_{\mathbb{Q}_{p}^{N}}}
\mathcal{F}\varphi\left(  \xi\right)  \overline{\mathcal{F}\psi\left(
\xi\right)  }d\mu\left(  \xi\right)  ,\text{ \ }\varphi,\psi\in\mathcal{D}%
_{\mathbb{R}}(\mathbb{Q}_{p}^{N}),
\]
where $\mu$ is the spectral measure of $W$. We denote by $\left\Vert
\cdot\right\Vert _{U}$ the corresponding norm. Then $U$ is a separable Hilbert
space (because $\mathcal{D}_{\mathbb{R}}(\mathbb{Q}_{p}^{N})$\ is separable)
that may contain distributions.

We fix \ a time interval $\left[  0,T\right]  $ and set $U_{T}:=L^{2}\left(
\left[  0,T\right]  ;U\right)  $. This set is equipped with the norm given by
\[
\left\Vert g\right\Vert _{U_{T}}^{2}:=%
{\displaystyle\int\nolimits_{0}^{T}}
\left\Vert g\left(  s\right)  \right\Vert _{U}^{2}ds.
\]
We now associate a cylindrical Wiener process to $W$ as follows. A direct
calculation using (\ref{Eq2}) shows that the generalized Gaussian random field
$W\left(  \varphi\right)  $ is a random linear functional, in the sense that
$W\left(  a\varphi+b\psi\right)  =aW\left(  \varphi\right)  +bW\left(
\psi\right)  $, almost surely, and $\varphi\rightarrow W\left(  \varphi
\right)  $ is an isometry from $\left(  \mathcal{I}(\left[  0,T\right]
)\otimes_{\text{alg}}\mathcal{D}_{\mathbb{R}}(\mathbb{Q}_{p}^{N}),\left\Vert
\cdot\right\Vert _{U_{T}}\right)  $ into $L^{2}\left(  \Omega,\mathcal{F}%
,P\right)  $. The following lemma identifies the completion of $\mathcal{I}%
(\left[  0,T\right]  )\otimes_{\text{alg}}\mathcal{D}_{\mathbb{R}}%
(\mathbb{Q}_{p}^{N})$ with respect to $\left\Vert \cdot\right\Vert _{U_{T}}$.

\begin{lemma}
\label{lemma1}The space $\mathcal{I}(\left[  0,T\right]  )\otimes_{\text{alg}%
}\mathcal{D}_{\mathbb{R}}(\mathbb{Q}_{p}^{N})$ is dense in $U_{T}=L^{2}\left(
\left[  0,T\right]  ;U\right)  $ for $\left\Vert \cdot\right\Vert _{U_{T}}$.
\end{lemma}

\begin{proof}
Let $\mathcal{C}$ denote the closure of $\mathcal{I}(\left[  0,T\right]
)\otimes_{\text{alg}}\mathcal{D}_{\mathbb{R}}(\mathbb{Q}_{p}^{N})$ in $U_{T}$
for $\left\Vert \cdot\right\Vert _{U_{T}}$. Suppose that we are given
$\varphi_{1}\in L^{2}\left(  \left[  0,T\right]  ;\mathbb{R}\right)  $ and
$\varphi_{2}\in\mathcal{D}_{\mathbb{R}}(\mathbb{Q}_{p}^{N})$. We show that
$\varphi_{1}\varphi_{2}\in\mathcal{C}$. Indeed, let $\left(  \varphi
_{1}^{\left(  n\right)  }\right)  _{n}\subset\mathcal{I}(\mathbb{R})$ such
that, for all n, the support of $\varphi_{1}^{\left(  n\right)  }$ is
contained in $\left[  0,T\right]  $ and $\varphi_{1}^{\left(  n\right)
}\rightarrow\varphi_{1}$ in $L^{2}\left(  \left[  0,T\right]  ;\mathbb{R}%
\right)  $. Now $\varphi_{1}^{\left(  n\right)  }\varphi_{2}\in\mathcal{I}%
(\left[  0,T\right]  )\otimes_{\text{alg}}\mathcal{D}_{\mathbb{R}}%
(\mathbb{Q}_{p}^{N})\subset\mathcal{C}$ and $\varphi_{1}^{\left(  n\right)
}\varphi_{2}$ $\underrightarrow{\left\Vert \cdot\right\Vert _{U_{T}}}$
$\varphi_{1}\varphi_{2}$, therefore $\varphi_{1}\varphi_{2}\in\mathcal{C}$.

Suppose that $\varphi\in U_{T}$. We show that $\varphi\in\mathcal{C}$. Indeed,
let $\left(  e_{j}\right)  _{j}$ be a complete orthonormal basis of $U$ with
$e_{j}\in\mathcal{D}_{\mathbb{R}}(\mathbb{Q}_{p}^{N})$ for all $j$. Then,
since $\varphi\left(  s\right)  \in U$ for any $s\in\left[  0,T\right]  $,%
\[
\left\Vert \varphi\right\Vert _{U_{T}}^{2}=%
{\displaystyle\int\nolimits_{0}^{T}}
\left\Vert \varphi\left(  s\right)  \right\Vert _{U}^{2}ds=%
{\displaystyle\sum\nolimits_{j=1}^{+\infty}}
{\displaystyle\int\nolimits_{0}^{T}}
\left\langle \varphi\left(  s\right)  ,e_{j}\right\rangle _{U}^{2}ds.
\]
We now note that for any $j\geq1$, the function $s\rightarrow\left\langle
\varphi\left(  s\right)  ,e_{j}\right\rangle _{U}$ belongs to $L^{2}\left(
\left[  0,T\right]  ;\mathbb{R}\right)  $. Thus, by the above considerations,%
\[
\varphi^{\left(  n\right)  }\left(  \cdot\right)  :=%
{\displaystyle\sum\nolimits_{j=1}^{n}}
\left\langle \varphi\left(  \cdot\right)  ,e_{j}\right\rangle _{U}e_{j}%
\in\mathcal{C}\text{.}%
\]
Finally, since \ $\lim_{n\rightarrow+\infty}\left\Vert \varphi-\varphi
^{\left(  n\right)  }\right\Vert _{U_{T}}^{2}=0$, we conclude that $\varphi
\in\mathcal{C}$.
\end{proof}

By using the above lemma, we can extend $W$ to $U_{T}$ following the standard
methods for extending an isometry. This establishes the following result.

\begin{proposition}
\label{Prop1}For $t\geq0$ and $\varphi\in U$, set $W_{t}\left(  \varphi
\right)  :=W\left(  1_{\left[  0,t\right]  }\left(  \cdot\right)
\varphi\left(  {\small \star}\right)  \right)  $. Then the process $W=\left\{
W_{t}\left(  \varphi\right)  ,t\geq0,\varphi\in U\right\}  $ is a cylindrical
Wiener process as in Definition \ref{Def1}, with $V$ there replaced by $U$
\ and $Q=I_{V}$. In particular, for any $\varphi\in U$, $\left\{  W_{t}\left(
\varphi\right)  ,t\geq0\right\}  $ is a Brownian motion with variance
$t\left\Vert \varphi\right\Vert _{U}$ and for all $s,t\geq0$ and $\varphi
,\psi\in U$, $E\left(  W_{t}\left(  \varphi\right)  W_{t}\left(  \psi\right)
\right)  =\left(  s\wedge t\right)  \left\langle \varphi,\psi\right\rangle
_{U}$.
\end{proposition}

\begin{remark}
\label{StochasticIntegrals}This proposition allow us to use the stochastic
integration defined in Section \ref{StochasticInt}. This defines the
stochastic integral $g\cdot W$ for all $g\in L^{2}\left(  \Omega\times\left[
0,T\right]  ;U\right)  $. In order to use the stochastic integral of Section
\ref{StochasticInt}, let $\left(  e_{j}\right)  _{j}\subset\mathcal{D}%
_{\mathbb{R}}(\mathbb{Q}_{p}^{N})$ be a complete orthonormal basis of $U$, and
consider the cylindrical Wiener process $\left\{  W_{t}(\varphi)\right\}  $
defined in Proposition \ref{Prop1}. For any predictable process $g$ in
$L^{2}\left(  \Omega\times\left[  0,T\right]  ;U\right)  $, the stochastic
integral with respect to $W$ is%
\[
g\cdot W=%
{\displaystyle\int\nolimits_{0}^{T}}
g_{s}dW_{s}:=%
{\displaystyle\sum\nolimits_{j=1}^{+\infty}}
{\displaystyle\int\nolimits_{0}^{T}}
\left\langle g_{s},e_{j}\right\rangle _{U}dW_{s}\left(  e_{j}\right)  \text{,
}%
\]
and the isometry property is given by%
\begin{equation}
E\left(  \left(  g\cdot W\right)  ^{2}\right)  =E\left(  \left(
{\displaystyle\int\nolimits_{0}^{T}}
g_{s}dW_{s}\right)  ^{2}\right)  =E\left(
{\displaystyle\int\nolimits_{0}^{T}}
\left\Vert g_{s}\right\Vert _{U}^{2}ds\right)  . \label{Eq5}%
\end{equation}
We also use the notation $\int_{0}^{T}\int_{\mathbb{Q}_{p}^{N}}g(s,y)W(ds,dy)$
instead of $\int_{0}^{T}g_{s}dW_{s}$. In later sections we will also use the
notation $E\left(  \int_{0}^{T}\left\Vert g\left(  s\right)  \right\Vert
_{U}^{2}ds\right)  $ for $E\left(  \left(  g\cdot W\right)  ^{2}\right)  $.
\end{remark}

\subsection{The spectral measure}

Recall that $\mu$\ is the spectral measure of $W$. In the following we use a
function $\Gamma$ satisfying the following hypothesis:

\textbf{Hypothesis A. }The function $\Gamma$ is defined on $\mathbb{R}%
_{+}:=\left[  0,+\infty\right)  $ with values in $\mathcal{D}_{\mathbb{R}%
}^{\prime}(\mathbb{Q}_{p}^{N})$ such that, for all $t>0$, $\Gamma\left(
t\right)  $ is a positive distribution satisfying%
\begin{equation}%
{\displaystyle\int\nolimits_{0}^{T}}
dt%
{\displaystyle\int\nolimits_{\mathbb{Q}_{p}^{N}}}
\left\vert \mathcal{F}\Gamma\left(  t\right)  \left(  \xi\right)  \right\vert
^{2}d\mu\left(  \xi\right)  <+\infty, \label{condition1}%
\end{equation}
and $\Gamma$ is associated with a measure $\Gamma\left(  t,d^{N}x\right)  $
such that, for all $T>0$,
\begin{equation}
\sup_{0\leq t\leq T}\Gamma\left(  t,\mathbb{Q}_{p}^{N}\right)  <+\infty.
\label{condition2}%
\end{equation}

We\ now set $\Gamma\left(  t,x\right)  =\mathcal{F}_{\xi\rightarrow x}%
^{-1}\left(  e^{-t\left\vert a\left(  \xi\right)  \right\vert _{p}^{\beta}%
}\right)  $, for $t>0$, and $\Gamma\left(  0,x\right)  :=\delta\left(
x\right)  $, i.e $\Gamma$ is the fundamental solution of (\ref{CauchyProblem}%
), and since $\Gamma\left(  t,x\right)  \in L^{1}\left(  \mathbb{Q}_{p}%
^{N},d^{N}x\right)  $ for $t>0$, it defines an element of $\mathcal{D}%
_{\mathbb{R}}^{\prime}(\mathbb{Q}_{p}^{N})$. In addition, $\Gamma\left(
t,d^{N}x\right)  :=\Gamma\left(  t,x\right)  d^{N}x$, and by Theorem
\ref{Thm1B} (ii) and (v),%

\[
\sup_{0\leq t\leq T}\Gamma\left(  t,\mathbb{Q}_{p}^{N}\right)  =\sup_{0<t\leq
T}%
{\displaystyle\int\nolimits_{\mathbb{Q}_{p}^{N}}}
\Gamma\left(  t,x\right)  d^{N}x=1.
\]
Hence $\Gamma\left(  t,d^{N}x\right)  $ satisfies (\ref{condition2}).

\begin{remark}
If $H\left(  t,x\right)  $ is a function on $\mathbb{R}\times\mathbb{Q}%
_{p}^{N}$, we use $H\left(  t\right)  $ instead of $H\left(  t,\cdot\right)
$. If $G\left(  t,x,\omega\right)  $ is a function on $\mathbb{R}%
\times\mathbb{Q}_{p}^{N}\times\Omega$, we use $G\left(  t,x\right)  $ instead
of $G\left(  t,x,\omega\right)  $, as it is customary in probability, in
certain special cases we will use $G\left(  t,x\right)  \left(  \omega\right)
$.
\end{remark}

On the other hand, by using Fubini's Theorem and inequality (\ref{elliptic}),
it is easy to check that condition (\ref{condition1}) is equivalent to
\begin{equation}%
{\displaystyle\int\nolimits_{\mathbb{Q}_{p}^{N}}}
\text{ }\frac{d\mu\left(  \xi\right)  }{\max\left(  1,\left\Vert
\xi\right\Vert _{p}\right)  ^{d\beta}}<+\infty. \label{condition1A}%
\end{equation}

\begin{lemma}
\label{lemma4}With the notation of Lemma \ref{lemma0}, assuming $\int
_{\mathbb{Q}_{p}^{N}}f(x)K_{d\beta}\left(  x\right)  d^{N}x<+\infty$ and
(\ref{condition1A}), we have%
\[%
{\displaystyle\int\nolimits_{\mathbb{Q}_{p}^{N}}}
\frac{d\mu\left(  \xi\right)  }{\max\left(  1,\left\Vert \xi\right\Vert
_{p}\right)  ^{d\beta}}=\int_{\mathbb{Q}_{p}^{N}}f(x)K_{d\beta}\left(
x\right)  d^{N}x.
\]

\end{lemma}

\begin{proof}
Set
\begin{equation}
\delta_{n}\left(  x\right)  :=p^{Nn}\Omega\left(  p^{n}\left\Vert x\right\Vert
_{p}\right)  \text{, for }n\in\mathbb{N}. \label{Delta_n}%
\end{equation}
Then $\int_{\mathbb{Q}_{p}^{N}}\delta_{n}\left(  x\right)  d^{N}x=1$ for any
$n$, $\delta_{n}$ $\underrightarrow{\mathcal{D}^{\prime}}$ $\delta$ and
$\mathcal{F}\delta_{n}\left(  \xi\right)  =\Omega\left(  p^{-n}\left\Vert
\xi\right\Vert _{p}\right)  $ $\underrightarrow{\text{{\tiny pointwise}}}$
$1$. Notice that $\left(  K_{d\beta}\ast\delta_{n}\right)  \left(  x\right)
=K_{d\beta}\left(  x\right)  $, for $x\in\mathbb{Q}_{p}^{N}\smallsetminus
\left\{  0\right\}  $ and for any $n>N(x)$, since $K_{d\beta}$ is radial, then%
\begin{equation}
f(x)\left(  K_{d\beta}\ast\delta_{n}\right)  \left(  x\right)
=\ f(x)K_{d\beta}\left(  x\right)  \text{, for }x\in\mathbb{Q}_{p}%
^{N}\smallsetminus\left\{  0\right\}  \text{ and }n\text{ big enough.}
\label{AE}%
\end{equation}

Now, by the Riesz-Markov-Kakutani Theorem, $\mu$ is an element of
$\mathcal{D}^{\prime}\left(  \mathbb{Q}_{p}^{N}\right)  $ and since
$K_{d\beta}\ast\delta_{n}\in\mathcal{D}\left(  \mathbb{Q}_{p}^{N}\right)  $,
we have%
\begin{align*}
\left(  \mu,\mathcal{F}\left(  K_{d\beta}\ast\delta_{n}\right)  \right)   &
=\left(  \mathcal{F}\mu,K_{d\beta}\ast\delta_{n}\right)  =\left(  f,K_{d\beta
}\ast\delta_{n}\right) \\
&  =\int_{\mathbb{Q}_{p}^{N}}f(x)\left(  K_{d\beta}\ast\delta_{n}\right)
\left(  x\right)  d^{N}x.
\end{align*}

Then by applying the Dominated Convergence Theorem and using (\ref{AE})\ and
the hypothesis $\int_{\mathbb{Q}_{p}^{N}}f(x)K_{d\beta}\left(  x\right)
d^{N}x<+\infty$, we get
\[
\lim_{n\rightarrow+\infty}\int_{\mathbb{Q}_{p}^{N}}f(x)\left(  K_{d\beta}%
\ast\delta_{n}\right)  \left(  x\right)  d^{N}x=\int_{\mathbb{Q}_{p}^{N}%
}f(x)K_{d\beta}\left(  x\right)  d^{N}x.
\]
On the other hand, by the Riesz-Markov-Kakutani Theorem,%
\begin{align*}
\left(  \mu,\mathcal{F}\left(  K_{d\beta}\ast\delta_{n}\right)  \right)   &
=\left(  \mu,\frac{\left(  \mathcal{F}\delta_{n}\right)  \left(  \xi\right)
}{\max\left(  1,\left\Vert \xi\right\Vert _{p}\right)  ^{d\beta}}\right) \\
&  =\int_{\mathbb{Q}_{p}^{N}}\frac{\left(  \mathcal{F}\delta_{n}\right)
\left(  \xi\right)  }{\max\left(  1,\left\Vert \xi\right\Vert _{p}\right)
^{d\beta}}d\mu\left(  \xi\right)  ,
\end{align*}
now, by the Dominated Convergence Theorem and Hypothesis (\ref{condition1A}),%
\[
\lim_{n\rightarrow+\infty}\int_{\mathbb{Q}_{p}^{N}}\frac{\left(
\mathcal{F}\delta_{n}\right)  \left(  \xi\right)  }{\max\left(  1,\left\Vert
\xi\right\Vert _{p}\right)  ^{d\beta}}d\mu\left(  \xi\right)  =\int
_{\mathbb{Q}_{p}^{N}}\frac{d\mu\left(  \xi\right)  }{\max\left(  1,\left\Vert
\xi\right\Vert _{p}\right)  ^{d\beta}}.
\]

\end{proof}

From Lemmas \ref{lemma0}-\ref{lemma4}, we obtain the following result:

\begin{theorem}
\label{Thm1}%
\begin{multline*}%
{\displaystyle\int\nolimits_{\mathbb{Q}_{p}^{N}}}
\frac{d\mu\left(  \xi\right)  }{\max\left(  1,\left\Vert \xi\right\Vert
_{p}\right)  ^{d\beta}}<+\infty\Leftrightarrow\\
\left\{
\begin{array}
[c]{lll}%
\frac{1-p^{-d\beta}}{1-p^{d\beta-N}}%
{\displaystyle\int\nolimits_{\left\Vert x\right\Vert _{p}\leq1}}
\left(  \left\Vert x\right\Vert _{p}^{d\beta-N}-p^{d\beta-N}\right)  f\left(
x\right)  d^{N}x<+\infty & \text{if} & d\beta\neq N\\
&  & \\
\left(  1-p^{-N}\right)
{\displaystyle\int\nolimits_{\left\Vert x\right\Vert _{p}\leq1}}
\log_{p}\left(  \frac{p}{\left\Vert x\right\Vert _{p}}\right)  f\left(
x\right)  d^{N}x<+\infty & \text{if} & d\beta=N.
\end{array}
\right.
\end{multline*}

\end{theorem}

\subsection{\label{Integrands}Examples of integrands}

The main examples of integrands are provided by the following result:

\begin{proposition}
\label{Prop2}Assume that $\Gamma$ satisfies Hypothesis A. Let
\[
Y=\left\{  Y(t,x),\left(  t,x\right)  \in\left[  0,T\right]  \times
\mathbb{Q}_{p}^{N}\right\}
\]
be a predictable process such that%
\[
C_{Y}:=\sup_{\left(  t,x\right)  \in\left[  0,T\right]  \times\mathbb{Q}%
_{p}^{N}}E\left(  \left\vert Y\left(  t,x\right)  \right\vert ^{2}\right)
<+\infty.
\]
Then, the random element $G=G\left(  t,x\right)  =Y\left(  t,x\right)
\Gamma\left(  t,x\right)  $ is a predictable process with values in
$L^{2}\left(  \Omega\times\left[  0,T\right]  ;U\right)  $. Moreover,
\begin{align*}
E\left(  \left\Vert G\right\Vert _{U_{T}}^{2}\right)   &  =E\left[
{\displaystyle\int\nolimits_{0}^{T}}
{\displaystyle\int\nolimits_{\mathbb{Q}_{p}^{N}}}
\left\vert \mathcal{F}\left(  \Gamma\left(  t\right)  Y\left(  t\right)
\left(  \xi\right)  \right)  \right\vert ^{2}d\mu\left(  \xi\right)  dt\right]
\\
&  \leq C_{Y}%
{\displaystyle\int\nolimits_{0}^{T}}
{\displaystyle\int\nolimits_{\mathbb{Q}_{p}^{N}}}
\left\vert \mathcal{F}\left(  \Gamma\left(  t\right)  Y\left(  t\right)
\left(  \xi\right)  \right)  \right\vert ^{2}d\mu\left(  \xi\right)  dt
\end{align*}
and
\begin{equation}
E\left(  \left\vert G\cdot W\right\vert ^{2}\right)  \leq%
{\displaystyle\int\nolimits_{0}^{T}}
\left(  \sup_{x\in\mathbb{Q}_{p}^{N}}E\left(  \left\vert Y\left(  s,x\right)
\right\vert ^{2}\right)  \right)
{\displaystyle\int\nolimits_{\mathbb{Q}_{p}^{N}}}
\left\vert \mathcal{F}\left(  \Gamma\left(  s\right)  \left(  \xi\right)
\right)  \right\vert ^{2}d\mu\left(  \xi\right)  ds. \label{Key_inequality}%
\end{equation}

\end{proposition}

\begin{remark}
The integral of $G$ with respect to $W$ will be also denoted by
\[
G\cdot W=%
{\displaystyle\int\nolimits_{0}^{T}}
{\displaystyle\int\nolimits_{\mathbb{Q}_{p}^{N}}}
\Gamma\left(  s,y\right)  Y\left(  s,y\right)  W(ds,d^{N}y).
\]

\end{remark}

\begin{proof}
The proof will be accomplished through several steps.

\textbf{§1.} \textbf{Assertion A}:$\ G(t)\in L^{1}\left(  \mathbb{Q}_{p}%
^{N}\right)  $, for $t\in\left(  0,T\right]  $ a.s.

Indeed, by the Hölder inequality,%
\[%
{\displaystyle\int\nolimits_{\mathbb{Q}_{p}^{N}}}
E\left(  \left\vert Y\left(  t,x\right)  \right\vert ^{2}\right)
\Gamma\left(  t,x\right)  d^{N}x\leq C_{Y}\left\Vert \Gamma\left(  t\right)
\right\Vert _{L^{1}\left(  \mathbb{Q}_{p}^{N}\right)  }\text{, for }%
t\in\left(  0,T\right]  \text{,}%
\]
cf. Theorem \ref{Thm1B} (iii). Hence,
\[%
{\displaystyle\int\nolimits_{\Omega}}
{\displaystyle\int\nolimits_{\mathbb{Q}_{p}^{N}}}
\left\vert Y\left(  t,x\right)  \left(  \omega\right)  \right\vert ^{2}%
\Gamma\left(  t,x\right)  d^{N}xdP\left(  \omega\right)  <+\infty,
\]
and by Fubini's Theorem, $\left\vert Y\left(  t,x\right)  \right\vert
^{2}\Gamma\left(  t,x\right)  \in L^{1}\left(  \mathbb{Q}_{p}^{N}\right)  $,
for $t\in\left(  0,T\right]  $ a.s. Now,%
\begin{multline*}%
{\displaystyle\int\nolimits_{\mathbb{Q}_{p}^{N}}}
\left\vert Y\left(  t,x\right)  \right\vert \Gamma\left(  t,x\right)  d^{N}x=%
{\displaystyle\int\nolimits_{\left\vert Y\left(  t,x\right)  \right\vert >1}}
\left\vert Y\left(  t,x\right)  \right\vert \Gamma\left(  t,x\right)  d^{N}x\\
+%
{\displaystyle\int\nolimits_{\left\vert Y\left(  t,x\right)  \right\vert
\leq1}}
\left\vert Y\left(  t,x\right)  \right\vert \Gamma\left(  t,x\right)  d^{N}x\\
\leq%
{\displaystyle\int\nolimits_{\mathbb{Q}_{p}^{N}}}
\left\vert Y\left(  t,x\right)  \right\vert ^{2}\Gamma\left(  t,x\right)
d^{N}x+%
{\displaystyle\int\nolimits_{\mathbb{Q}_{p}^{N}}}
\Gamma\left(  t,x\right)  d^{N}x<+\infty\text{ for }t\in\left(  0,T\right]
\text{ a.s.}%
\end{multline*}

By using the above reasoning, one verifies that $\int_{0}^{T}\int
_{\mathbb{Q}_{p}^{N}}\left\vert Y\left(  t,x\right)  \right\vert h\left(
t,x\right)  d^{N}xdt<+\infty$, for every $h\in L^{1}\left(  \left[
0,T\right]  \times\mathbb{Q}_{p}^{N}\right)  $ satisfying $h\geq0$, therefore
\begin{equation}
Y\left(  t,x\right)  \in L^{\infty}\left(  \left[  0,T\right]  \times
\mathbb{Q}_{p}^{N}\right)  \text{ a.s.} \label{BoundY}%
\end{equation}

As a consequence of Assertion A, we have $G(t)\in\mathcal{D}_{\mathbb{R}%
}^{\prime}\left(  \mathbb{Q}_{p}^{N}\right)  $, for $t\in\left(  0,T\right]  $
a.s. We now proceed to regularize this distribution. \ We set $\Delta
_{l}\left(  x\right)  :=\Omega\left(  p^{-l}\left\Vert x\right\Vert
_{p}\right)  $, $\delta_{k}\left(  x\right)  =p^{kN}\Omega\left(
p^{k}\left\Vert x\right\Vert _{p}\right)  $ for $k,l\in\mathbb{N}$. Then
$\int\delta_{k}\left(  x\right)  d^{N}x=1$, $\mathcal{F}\left(  \Delta
_{k}\right)  =\delta_{k}$, $\delta_{k}$ $\underrightarrow{\mathcal{D}%
_{\mathbb{R}}^{\prime}}$ $\delta$ (Dirac distribution), and $\Delta_{l}$
$\underrightarrow{\text{pointwise}}$ $1$, as before.

We also set $G_{k,l}(t):=\left(  \Delta_{l}Y\left(  t\right)  \Gamma\left(
t\right)  \right)  \ast\delta_{k}$, $k,l\in\mathbb{N}$, $t\in\left(
0,T\right]  $, and $\Gamma_{k}\left(  t\right)  :=\Gamma\left(  t\right)
\ast\delta_{k}$. Then, $G_{k,l}(t)\in\mathcal{D}_{\mathbb{R}}\left(
\mathbb{Q}_{p}^{N}\right)  $ for $t\in\left(  0,T\right]  $, more precisely,
\[
G_{k,l}(t)=\sum_{j}c_{j}\left(  t;k,l\right)  \Omega\left(  p^{m\left(
k\right)  }\left\Vert x-x_{j}\right\Vert _{p}\right)  .
\]
Now, since
\begin{align*}
\left\vert \left(  \Delta_{l}Y\left(  t\right)  \Gamma\left(  t\right)
\right)  \ast\delta_{k}\right\vert  &  \leq\left\vert Y\left(  t\right)
\right\vert \Gamma\left(  t\right)  \ast\delta_{k}\leq\left\Vert Y\right\Vert
_{L^{\infty}\left(  \left[  0,T\right]  \times\mathbb{Q}_{p}^{N}\right)
}\left\Vert \Gamma\left(  t\right)  \ast\delta_{k}\right\Vert _{L^{1}\left(
\mathbb{Q}_{p}^{N}\right)  }\\
&  \leq p^{Nk}\left\Vert Y\right\Vert _{L^{\infty}\left(  \left[  0,T\right]
\times\mathbb{Q}_{p}^{N}\right)  }\text{\ \ a.s.,}%
\end{align*}
cf. (\ref{BoundY}), we have $c_{j}\left(  t;k,l\right)  \in L^{\infty}\left(
\left[  0,T\right]  \right)  $ a.s. Therefore%
\begin{equation}
G_{k,l}(t)\in L^{2}\left(  \left[  0,T\right]  \right)  \otimes_{\text{alg}%
}\mathcal{D}_{\mathbb{R}}\left(  \mathbb{Q}_{p}^{N}\right)  \text{ a.s.}
\label{BoundGkl}%
\end{equation}

\textbf{§2. A bound for }$\sup_{k,l\geq1}E\left\Vert G_{k,l}\right\Vert
_{U_{T}}^{2}$.

By using the definition of the convolution and the uniform bound for the
square moments of $Y$, we get:%
\begin{align*}
\sup_{k,l\geq1}E\left\Vert G_{k,l}\right\Vert _{U_{T}}^{2}  &  =\sup
_{k,l\geq1}E%
{\displaystyle\int\nolimits_{0}^{T}}
{\displaystyle\int\nolimits_{\mathbb{Q}_{p}^{N}}}
{\displaystyle\int\nolimits_{\mathbb{Q}_{p}^{N}}}
G_{k,l}\left(  t,x\right)  f\left(  x-y\right)  G_{k,l}\left(  t,y\right)
d^{N}xd^{N}ydt\\
&  =\sup_{k,l\geq1}E%
{\displaystyle\int\nolimits_{0}^{T}}
{\displaystyle\int\nolimits_{\mathbb{Q}_{p}^{N}}}
{\displaystyle\int\nolimits_{\mathbb{Q}_{p}^{N}}}
\left[
{\displaystyle\int\nolimits_{\mathbb{Q}_{p}^{N}}}
\Delta_{l}\left(  z\right)  Y\left(  t,z\right)  \Gamma\left(  t,z\right)
\delta_{k}\left(  x-z\right)  d^{N}z\right] \\
&  \times f\left(  x-y\right)  \left[
{\displaystyle\int\nolimits_{\mathbb{Q}_{p}^{N}}}
\Delta_{l}\left(  z^{\prime}\right)  Y\left(  t,z^{\prime}\right)
\Gamma\left(  t,z^{\prime}\right)  \delta_{k}\left(  y-z^{\prime}\right)
d^{N}z^{\prime}\right]  d^{N}xd^{N}ydt\\
&  \leq C_{Y}\sup_{k\geq1}%
{\displaystyle\int\nolimits_{0}^{T}}
{\displaystyle\int\nolimits_{\mathbb{Q}_{p}^{N}}}
{\displaystyle\int\nolimits_{\mathbb{Q}_{p}^{N}}}
\Gamma_{k}\left(  t,x\right)  f\left(  x-y\right)  \Gamma_{k}\left(
t,y\right)  d^{N}xd^{N}ydt\\
&  =C_{Y}\sup_{k\geq1}%
{\displaystyle\int\nolimits_{0}^{T}}
\left\vert \mathcal{F}\Gamma_{k}\left(  t\right)  \left(  \xi\right)
\right\vert ^{2}d\mu\left(  \xi\right)  dt\leq C_{Y}%
{\displaystyle\int\nolimits_{0}^{T}}
\left\vert \mathcal{F}\Gamma\left(  t\right)  \left(  \xi\right)  \right\vert
^{2}d\mu\left(  \xi\right)  dt,
\end{align*}

because $\mathcal{F}\Gamma_{k}\left(  t\right)  \left(  \xi\right)
=\mathcal{F}\Gamma\left(  t\right)  \left(  \xi\right)  \cdot\Omega\left(
p^{-k}\left\Vert \xi\right\Vert _{p}\right)  \leq\mathcal{F}\Gamma\left(
t\right)  \left(  \xi\right)  $. Therefore%
\begin{equation}
\sup_{k,l\geq1}E\left(  \left\Vert G_{k,l}\right\Vert _{U_{T}}^{2}\right)
\leq C_{Y}%
{\displaystyle\int\nolimits_{0}^{T}}
\left\vert \mathcal{F}\Gamma\left(  t\right)  \left(  \xi\right)  \right\vert
^{2}d\mu\left(  \xi\right)  dt<+\infty\text{ } \label{Bound}%
\end{equation}
for $t\in\left(  0,T\right]  $ a.s.

As a consequence, we get that $G_{k,l}\in L^{2}\left(  \Omega\times\left[
0,T\right]  ;U\right)  $ for $k,l\geq1$, since by (\ref{BoundGkl}),
$G_{k,l}(t)\in L^{2}\left(  \left[  0,T\right]  \right)  \otimes_{\text{alg}%
}\mathcal{D}_{\mathbb{R}}\left(  \mathbb{Q}_{p}^{N}\right)  $ a.s.

\textbf{§3. }$\lim_{k\rightarrow+\infty}G_{k,l}\in L^{2}\left(  \Omega
\times\left[  0,T\right]  ;U\right)  $.

We set $G_{l}(t):=\Delta_{l}Y\left(  t\right)  \Gamma\left(  t\right)  $,
$l\in\mathbb{N}$, $t\in\left(  0,T\right]  $. By using the reasoning given in
Paragraph 2, we get%
\begin{equation}
E%
{\displaystyle\int\nolimits_{0}^{T}}
{\displaystyle\int\nolimits_{\mathbb{Q}_{p}^{N}}}
\left\vert \mathcal{F}G_{l}\left(  t\right)  \left(  \xi\right)  \right\vert
^{2}d\mu\left(  \xi\right)  dt\leq C_{Y}%
{\displaystyle\int\nolimits_{0}^{T}}
{\displaystyle\int\nolimits_{\mathbb{Q}_{p}^{N}}}
\left\vert \mathcal{F}\Gamma\left(  t\right)  \left(  \xi\right)  \right\vert
^{2}d\mu\left(  \xi\right)  dt<+\infty, \label{Bound2}%
\end{equation}
for any $l\in\mathbb{N}$.

We now assert that $G_{k,l}$ $\underrightarrow{U_{T}}$ $G_{l}$ as
$k\rightarrow+\infty$. Indeed,
\begin{align*}
&  E%
{\displaystyle\int\nolimits_{0}^{T}}
{\displaystyle\int\nolimits_{\mathbb{Q}_{p}^{N}}}
\left\vert \mathcal{F}G_{k,l}\left(  t\right)  \left(  \xi\right)
-\mathcal{F}G_{l}\left(  t\right)  \left(  \xi\right)  \right\vert ^{2}%
d\mu\left(  \xi\right)  dt\\
&  =E%
{\displaystyle\int\nolimits_{0}^{T}}
{\displaystyle\int\nolimits_{\mathbb{Q}_{p}^{N}}}
\left\vert \mathcal{F}G_{l}\left(  t\right)  \left(  \xi\right)  \right\vert
^{2}\left\vert \Delta_{k}\left(  \xi\right)  -1\right\vert ^{2}d\mu\left(
\xi\right)  dt\rightarrow0\text{ as }k\rightarrow+\infty,
\end{align*}
by the Dominated Convergence Theorem and (\ref{Bound2}). Hence $G_{l}\in
L^{2}\left(  \Omega\times\left[  0,T\right]  ;U\right)  $ and by
(\ref{Bound2}),
\begin{equation}
\sup_{l\geq1}E\left(  \left\Vert G_{l}\right\Vert _{U_{T}}^{2}\right)  \leq
C_{Y}%
{\displaystyle\int\nolimits_{0}^{T}}
\left\vert \mathcal{F}\Gamma\left(  t\right)  \left(  \xi\right)  \right\vert
^{2}d\mu\left(  \xi\right)  dt. \label{SupEGl}%
\end{equation}

\textbf{§4. }$G_{l}$ $\underrightarrow{U_{T}}$ $G$, i.e. $\lim_{l\rightarrow
+\infty}E\left(  \left\Vert G-G_{l}\right\Vert _{U_{T}}^{2}\right)  =0$.

Indeed,
\begin{gather*}
E\left(  \left\Vert G-G_{l}\right\Vert _{U_{T}}^{2}\right)  =E%
{\displaystyle\int\nolimits_{0}^{T}}
{\displaystyle\int\nolimits_{\mathbb{Q}_{p}^{N}}}
\left\vert \mathcal{F}G\left(  t\right)  \left(  \xi\right)  -\mathcal{F}%
G_{l}\left(  t\right)  \left(  \xi\right)  \right\vert ^{2}d\mu\left(
\xi\right)  dt\\
\leq2E%
{\displaystyle\int\nolimits_{0}^{T}}
{\displaystyle\int\nolimits_{\mathbb{Q}_{p}^{N}}}
\left\vert \mathcal{F}G\left(  t\right)  \left(  \xi\right)  \right\vert
^{2}d\mu\left(  \xi\right)  dt+2E%
{\displaystyle\int\nolimits_{0}^{T}}
{\displaystyle\int\nolimits_{\mathbb{Q}_{p}^{N}}}
\left\vert \mathcal{F}G_{l}\left(  t\right)  \left(  \xi\right)  \right\vert
^{2}d\mu\left(  \xi\right)  dt\\
\leq4E%
{\displaystyle\int\nolimits_{0}^{T}}
{\displaystyle\int\nolimits_{\mathbb{Q}_{p}^{N}}}
\left\vert \mathcal{F}G\left(  t\right)  \left(  \xi\right)  \right\vert
^{2}d\mu\left(  \xi\right)  dt\\
\leq4C_{Y}%
{\displaystyle\int\nolimits_{0}^{T}}
{\displaystyle\int\nolimits_{\mathbb{Q}_{p}^{N}}}
\left\vert \mathcal{F}\Gamma\left(  t\right)  \left(  \xi\right)  \right\vert
^{2}d\mu\left(  \xi\right)  dt<+\infty,
\end{gather*}
the last inequality was obtained by using the reasoning given in Paragraph 2.
On the other hand,%
\[
E\left(  \left\Vert G-G_{l}\right\Vert _{U_{T}}^{2}\right)  =E%
{\displaystyle\int\nolimits_{0}^{T}}
{\displaystyle\int\nolimits_{\mathbb{Q}_{p}^{N}}}
\left\vert \mathcal{F}G\left(  t\right)  \left(  \xi\right)  -\mathcal{F}%
G\left(  t\right)  \left(  \xi\right)  \ast\delta_{l}\left(  \xi\right)
\right\vert ^{2}d\mu\left(  \xi\right)  dt.
\]
Now by using the Dominated Convergence Theorem and the fact that
\[
\lim_{l\rightarrow+\infty}\mathcal{F}G\left(  t\right)  \left(  \xi\right)
\ast\delta_{l}\left(  \xi\right)  =\mathcal{F}G\left(  t\right)  \left(
\xi\right)  \text{ almost everywhere,}%
\]
cf. \cite[Theorem 1.14]{Taibleson}, we get that $\lim_{l\rightarrow+\infty
}E\left(  \left\Vert G-G_{l}\right\Vert _{U_{T}}^{2}\right)  =0$, which
implies $G\in L^{2}\left(  \Omega\times\left[  0,T\right]  ;U\right)  $.

Moreover, we deduce that%
\begin{align*}
E\left(  \left\Vert G\right\Vert _{U_{T}}^{2}\right)   &  =E\left(
{\displaystyle\int\nolimits_{0}^{T}}
{\displaystyle\int\nolimits_{\mathbb{Q}_{p}^{N}}}
\left\vert \mathcal{F}G\left(  t\right)  \left(  \xi\right)  \right\vert
^{2}d\mu\left(  \xi\right)  dt\right)  =\lim_{l\rightarrow+\infty}E\left(
\left\Vert G_{l}\right\Vert _{U_{T}}^{2}\right) \\
&  \leq C_{Y}%
{\displaystyle\int\nolimits_{0}^{T}}
{\displaystyle\int\nolimits_{\mathbb{Q}_{p}^{N}}}
\left\vert \mathcal{F}\Gamma\left(  t\right)  \left(  \xi\right)  \right\vert
^{2}d\mu\left(  \xi\right)  dt,
\end{align*}

cf. (\ref{SupEGl}).

\textbf{§5. A bound for} $E\left(  \left\vert G\cdot W\right\vert ^{2}\right)
$.

The announced bound for $E\left(  \left\vert G\cdot W\right\vert ^{2}\right)
$ is obtained from (\ref{Eq5}) by using a reasoning similar to the one used in
Paragraph 2.
\end{proof}

\begin{remark}
\label{notaprocess}Let $Y$ be a process as in Proposition \ref{Prop2}.
Consider the processes of the form%
\[
\left\{  Y(t,x),\left(  t,x\right)  \in\left[  T_{0},T\right]  \times
\mathbb{Q}_{p}^{N}\right\}
\]
where $0\leq T_{0}<T$, then
\begin{equation}
E\left(  \left\vert G\cdot W\right\vert ^{2}\right)  \leq%
{\displaystyle\int\nolimits_{T_{0}}^{T}}
\left(  \sup_{x\in\mathbb{Q}_{p}^{N}}E\left(  \left\vert Y\left(  s,x\right)
\right\vert ^{2}\right)  \right)
{\displaystyle\int\nolimits_{\mathbb{Q}_{p}^{N}}}
\left\vert \mathcal{F}\left(  \Gamma\left(  s\right)  \left(  \xi\right)
\right)  \right\vert ^{2}d\mu\left(  \xi\right)  ds. \label{EGW_modified}%
\end{equation}

\end{remark}

\section{\label{Sect6}Stochastic Pseudodifferential Equations Driven by a
spatially homogeneous Noise}

In this section we introduce a new class of stochastic pseudodifferential
\ equations in $\mathbb{Q}_{p}^{N}$ driven by a spatially homogeneous noise,
more precisely, we study the following class of stochastic equations:%
\begin{equation}
\left\{
\begin{array}
[c]{l}%
\frac{\partial u}{\partial t}\left(  t,x\right)  +\boldsymbol{A}\left(
\partial,\beta\right)  u\left(  t,x\right)  =\sigma\left(  u\left(
t,x\right)  \right)  \overset{\cdot}{W}\left(  t,x\right)  +b\left(  u\left(
t,x\right)  \right) \\
\\
u\left(  0,x\right)  =u_{0}\left(  x\right)  ,\text{ }t\geq0,x\in
\mathbb{Q}_{p}^{N},
\end{array}
\right.  \label{IVP}%
\end{equation}
where the coefficients $\sigma$\ and $b$\ are real-valued functions and
$\overset{\cdot}{W}\left(  t,x\right)  $ is the formal notation for the
Gaussian random perturbation described in Section \ref{SectionNoise}.

Recall that we are working with a filtered probability space $\left(
\Omega,\mathcal{F},\left(  \mathcal{F}_{t}\right)  ,P\right)  $, where
$\left(  \mathcal{F}_{t}\right)  _{t}$ is a filtration generated by the
standard cylindrical Wiener process of Proposition \ref{Prop1}. We fix a time
horizon $T>0$.

\begin{definition}
A real-valued adapted stochastic process
\[
\left\{  u\left(  t,x\right)  ,\left(  t,x\right)  \in\left[  0,T\right]
\times\mathbb{Q}_{p}^{N}\right\}
\]
is \textit{a mild random field solution of} (\ref{IVP}), if the following
stochastic integral equation is satisfied:%
\begin{align}
u\left(  t,x\right)   &  =\left(  \Gamma\left(  t\right)  \ast u_{0}\right)
\left(  x\right)  +%
{\displaystyle\int\nolimits_{0}^{t}}
{\displaystyle\int\nolimits_{\mathbb{Q}_{p}^{N}}}
\Gamma\left(  t-s,x-y\right)  \sigma\left(  u\left(  s,y\right)  \right)
W\left(  ds,d^{N}y\right) \nonumber\\
&  +%
{\displaystyle\int\nolimits_{0}^{t}}
ds%
{\displaystyle\int\nolimits_{\mathbb{Q}_{p}^{N}}}
\Gamma\left(  s,y\right)  b\left(  u\left(  t-s,x-y\right)  \right)
d^{N}y,\text{ a.s.,} \label{SolutionIVP}%
\end{align}
for all $\left(  t,x\right)  \in\left[  0,T\right]  \times\mathbb{Q}_{p}^{N}$.
\end{definition}

The stochastic integral on the right-hand side of (\ref{SolutionIVP}) is \ as
defined in Remark \ref{StochasticIntegrals}. In particular, we need to assume
that for any $\left(  t,x\right)  $ the fundamental solution $\Gamma\left(
t-\cdot,x-\star\right)  $ satisfies Hypothesis A, and to require that
\[
s\rightarrow\Gamma\left(  t-s,x-\star\right)  \sigma\left(  u\left(
s,\star\right)  \right)  \text{, for }s\in\left[  0,t\right]  ,
\]
defines a predictable process taking values in the space $U$ such that%
\[
E\left(
{\displaystyle\int\nolimits_{0}^{t}}
\left\Vert \Gamma\left(  t-s,x-\star\right)  \sigma\left(  u\left(
s,\star\right)  \right)  \right\Vert _{U}^{2}ds\right)  <+\infty,
\]
see Section \ref{Integrands}. These assumptions will be satisfied by imposing
that $b$\ and $\sigma$ are Lipschitz continuous functions (see Theorem
\ref{Thm2}). The last integral on the right-hand side of (\ref{SolutionIVP})
is considered in the pathwise sense.

The aim of this section is to prove the existence and uniqueness of a mild
random field solution for stochastic integral equation (\ref{SolutionIVP}). We
are interested in solutions that are $L^{2}\left(  \Omega\right)  $-bounded
and $L^{2}\left(  \Omega\right)  $-continuous.

\begin{lemma}
\label{lemma5}Assume that $u_{0}:\mathbb{Q}_{p}^{N}\rightarrow\mathbb{R}$ is
measurable and bounded. Then
\[
\left(  t,x\right)  \rightarrow I_{0}\left(  t,x\right)  :=\left(
\Gamma\left(  t\right)  \ast u_{0}\right)  \left(  x\right)
\]
is continuous and $\sup_{\left(  t,x\right)  \in\left[  0,T\right]
\times\mathbb{Q}_{p}^{N}}\left\vert I_{0}\left(  t,x\right)  \right\vert
<+\infty$.
\end{lemma}

\begin{proof}
Notice that
\begin{equation}
\left\vert I_{0}\left(  t,x\right)  \right\vert \leq\left\{
\begin{array}
[c]{lll}%
\left\Vert u_{0}\right\Vert _{L^{\infty}}\left\Vert \Gamma\left(  t\right)
\right\Vert _{L^{1}} & \text{for} & t>0\\
&  & \\
\left\Vert u_{0}\right\Vert _{L^{\infty}} & \text{for} & t=0,
\end{array}
\right.  \label{EqA}%
\end{equation}
and
\begin{equation}
\sup_{\left(  t,x\right)  \in\left(  0,T\right]  \times\mathbb{Q}_{p}^{N}%
}\left\vert I_{0}\left(  t,x\right)  \right\vert \leq\left\Vert u_{0}%
\right\Vert _{L^{\infty}}\sup_{\left(  t,x\right)  \in\left(  0,T\right]
\times\mathbb{Q}_{p}^{N}}\left\Vert \Gamma\left(  t\right)  \right\Vert
_{L^{1}}=\left\Vert u_{0}\right\Vert _{L^{\infty}}. \label{EqB}%
\end{equation}
By combining (\ref{EqA})-(\ref{EqB}), we get $\sup_{\left(  t,x\right)
\in\left[  0,T\right]  \times\mathbb{Q}_{p}^{N}}\left\vert I_{0}\left(
t,x\right)  \right\vert \leq\left\Vert u_{0}\right\Vert _{L^{\infty}}$.

The continuity of $I_{0}\left(  t,x\right)  $ at a point of the form $\left(
t_{0},x_{0}\right)  $, with $t_{0}>0$, follows by the Dominated Convergence
Theorem and Theorem \ref{Thm1B} (vi). The continuity of $I_{0}\left(
t,x\right)  $ at $\left(  0,x_{0}\right)  $ is a consequence of the fact that
$\lim_{\left(  t,x\right)  \rightarrow\left(  0,x_{0}\right)  }I_{0}\left(
t,x\right)  =u_{0}\left(  x_{0}\right)  =I_{0}\left(  0,x_{0}\right)  $.
\end{proof}

\textbf{Hypothesis B. }Let $\Gamma$ be the fundamental solution of
(\ref{CauchyProblem}) as before. We assume that
\begin{equation}
\lim_{h\rightarrow0^{+}}%
{\displaystyle\int\nolimits_{0}^{T}}
{\displaystyle\int\nolimits_{\mathbb{Q}_{p}^{N}}}
\sup_{\left\vert r-t\right\vert <h}\left\vert \mathcal{F}\Gamma\left(
r\right)  \left(  \xi\right)  -\mathcal{F}\Gamma\left(  t\right)  \left(
\xi\right)  \right\vert ^{2}d\mu\left(  \xi\right)  dt=0. \label{LimF}%
\end{equation}

\begin{lemma}
\label{lemma5A}If $\int_{\mathbb{Q}_{p}^{N}}\left\Vert \xi\right\Vert
_{p}^{d\beta}d\mu\left(  \xi\right)  <+\infty$, then Hypothesis B holds.
Furthermore, the condition $\int_{\mathbb{Q}_{p}^{N}}\left\Vert \xi\right\Vert
_{p}^{d\beta}d\mu\left(  \xi\right)  <+\infty$ also implies (\ref{condition1}).
\end{lemma}

\begin{proof}
By applying the Mean Value Theorem to the function $e^{-t\left\vert a\left(
\xi\right)  \right\vert _{p}^{\beta}}$, we have%
\begin{align*}
\sup_{\left\vert r-t\right\vert <h}\left\vert \mathcal{F}\Gamma\left(
r\right)  \left(  \xi\right)  -\mathcal{F}\Gamma\left(  t\right)  \left(
\xi\right)  \right\vert ^{2}  &  \leq h^{2}\left\vert a\left(  \xi\right)
\right\vert _{p}^{2\beta}e^{-2\left(  t-h\right)  \left\vert a\left(
\xi\right)  \right\vert _{p}^{\beta}}\\
&  \leq C_{1}^{2\beta}h^{2}\left\Vert \xi\right\Vert _{p}^{2d\beta}%
e^{-2C_{0}^{\beta}\left(  t-h\right)  \left\Vert \xi\right\Vert _{p}^{d\beta}%
},
\end{align*}
cf. (\ref{elliptic}), and thus
\begin{align*}
&
{\displaystyle\int\nolimits_{0}^{T}}
{\displaystyle\int\nolimits_{\mathbb{Q}_{p}^{N}}}
\sup_{\left\vert r-t\right\vert <h}\left\vert \mathcal{F}\Gamma\left(
r\right)  \left(  \xi\right)  -\mathcal{F}\Gamma\left(  t\right)  \left(
\xi\right)  \right\vert ^{2}d\mu\left(  \xi\right)  dt\\
&  \leq h^{2}%
{\displaystyle\int\nolimits_{0}^{T}}
{\displaystyle\int\nolimits_{\mathbb{Q}_{p}^{N}}}
\left\Vert \xi\right\Vert _{p}^{2d\beta}e^{-2C_{0}^{\beta}\left(  t-h\right)
\left\Vert \xi\right\Vert _{p}^{d\beta}}d\mu\left(  \xi\right)  dt.
\end{align*}
In order to prove the result, it is sufficient to show that%
\begin{equation}
\lim_{h\rightarrow0^{+}}%
{\displaystyle\int\nolimits_{0}^{T}}
{\displaystyle\int\nolimits_{\mathbb{Q}_{p}^{N}}}
\left\Vert \xi\right\Vert _{p}^{2d\beta}e^{-2C_{0}^{\beta}\left(  t-h\right)
\left\Vert \xi\right\Vert _{p}^{d\beta}}d\mu\left(  \xi\right)  dt<+\infty.
\label{formula1}%
\end{equation}

Now, if
\begin{equation}%
{\displaystyle\int\nolimits_{0}^{T}}
{\displaystyle\int\nolimits_{\mathbb{Q}_{p}^{N}}}
\left\Vert \xi\right\Vert _{p}^{2d\beta}e^{-2C_{0}^{\beta}t\left\Vert
\xi\right\Vert _{p}^{d\beta}}d\mu\left(  \xi\right)  dt<+\infty,
\label{formula2}%
\end{equation}
then (\ref{formula1}) follows by applying the Dominated Convergence Theorem.
On the other hand, it is easy to check that $\int_{\mathbb{Q}_{p}^{N}%
}\left\Vert \xi\right\Vert _{p}^{d\beta}d\mu\left(  \xi\right)  <+\infty$
implies
\begin{equation}%
{\displaystyle\int\nolimits_{\mathbb{Q}_{p}^{N}}}
\left\Vert \xi\right\Vert _{p}^{2d\beta}%
{\displaystyle\int\nolimits_{0}^{T}}
e^{-2C_{0}^{\beta}t\left\Vert \xi\right\Vert _{p}^{d\beta}}dtd\mu\left(
\xi\right)  <+\infty, \label{formula3}%
\end{equation}
now (\ref{formula2}) follows \ from (\ref{formula3}) by using the Fubini
Theorem. The last assertion in the statement follows from the fact that
(\ref{condition1}) is equivalent to (\ref{condition1A}).
\end{proof}

\begin{theorem}
\label{Thm2}Assume that $b$, $\sigma$ are Lipschitz continuous functions,
$u_{0}$ is measurable and bounded function, and that
\[%
{\displaystyle\int\nolimits_{\mathbb{Q}_{p}^{N}}}
\left\Vert \xi\right\Vert _{p}^{d\beta}d\mu\left(  \xi\right)  <+\infty.
\]
Then, there exists a unique mild random field solution $\left\{  u\left(
t,x\right)  ,\left(  t,x\right)  \in\left[  0,T\right]  \times\mathbb{Q}%
_{p}^{N}\right\}  $ of (\ref{SolutionIVP}). Moreover, $u$ is $L^{2}\left(
\Omega\right)  $-continuous and
\begin{equation}
\sup_{\left(  t,x\right)  \in\left[  0,T\right]  \times\mathbb{Q}_{p}^{N}%
}E\left(  \left\vert u\left(  t,x\right)  \right\vert ^{2}\right)  <+\infty.
\label{EqExpe}%
\end{equation}

\end{theorem}

\begin{proof}
The proof involves similar techniques and ideas to those of \cite{Dalang},
\cite{Dalang-Quer}, \cite{M-S}.\ We use the following Picard iteration scheme:%
\begin{equation}
u^{0}\left(  t,x\right)  =I_{0}\left(  t,x\right)  , \label{U_0}%
\end{equation}%
\begin{align}
u^{n+1}(t,x)  &  =u^{0}\left(  t,x\right)  +%
{\displaystyle\int\nolimits_{0}^{t}}
{\displaystyle\int\nolimits_{\mathbb{Q}_{p}^{N}}}
\Gamma\left(  t-s,x-y\right)  \sigma\left(  u^{n}\left(  s,y\right)  \right)
W\left(  ds,d^{N}y\right) \nonumber\\
&  +%
{\displaystyle\int\nolimits_{0}^{t}}
{\displaystyle\int\nolimits_{\mathbb{Q}_{p}^{N}}}
b\left(  u^{n}\left(  t-s,x-y\right)  \right)  \Gamma\left(  s,y\right)
d^{N}yds\nonumber\\
&  =:u^{0}\left(  t,x\right)  +\mathcal{I}^{n}\left(  t,x\right)
+\mathcal{J}^{n}\left(  t,x\right)  , \label{U_n+1}%
\end{align}

for $n\in\mathbb{N}$.

The proof will be accomplished through several steps.

\textbf{§1. }$u^{n}\left(  t,x\right)  $ is a well-defined measurable process.

We prove by induction on $n$ that $\left\{  u^{n}\left(  t,x\right)  ,\left(
t,x\right)  \in\left[  0,T\right]  \times\mathbb{Q}_{p}^{N}\right\}  $ is a
well-defined\ measurable process satisfying
\begin{equation}
\sup_{\left(  t,x\right)  \in\left[  0,T\right]  \times\mathbb{Q}_{p}^{N}%
}E\left(  \left\vert u^{n}\left(  t,x\right)  \right\vert ^{2}\right)
<+\infty, \label{EqC}%
\end{equation}
for $n\in\mathbb{N}$. By Lemma \ref{lemma5}, $u^{0}\left(  t,x\right)  $
satisfies (\ref{EqC}), and the Lipschitz property of $\sigma$\ implies that
\[
\sup_{\left(  t,x\right)  \in\left[  0,T\right]  \times\mathbb{Q}_{p}^{N}%
}\sigma\left(  u^{0}\left(  t,x\right)  \right)  <+\infty.
\]
By Proposition \ref{Prop2}, the stochastic integral
\[
\mathcal{I}^{0}\left(  t,x\right)  =%
{\displaystyle\int\nolimits_{0}^{t}}
{\displaystyle\int\nolimits_{\mathbb{Q}_{p}^{N}}}
\Gamma\left(  t-s,x-y\right)  \sigma\left(  u^{0}\left(  s,y\right)  \right)
W\left(  ds,d^{N}y\right)
\]
is well-defined and
\begin{align}
E\left(  \left\vert \mathcal{I}^{0}\left(  t,x\right)  \right\vert
^{2}\right)   &  \leq C%
{\displaystyle\int\nolimits_{0}^{t}}
\sup_{z\in\mathbb{Q}_{p}^{N}}\left(  1+\left\vert u^{0}\left(  s,z\right)
\right\vert ^{2}\right)
{\displaystyle\int\nolimits_{\mathbb{Q}_{p}^{N}}}
\left\vert \mathcal{F}\Gamma\left(  t-s\right)  \left(  \xi\right)
\right\vert ^{2}d\mu\left(  \xi\right)  ds\nonumber\\
&  \leq C\sup_{\left(  s,z\right)  \in\left[  0,T\right]  \times\mathbb{Q}%
_{p}^{N}}\left(  1+\left\vert u^{0}\left(  s,z\right)  \right\vert
^{2}\right)
{\displaystyle\int\nolimits_{0}^{T}}
J\left(  s\right)  ds, \label{EqD}%
\end{align}
where%
\[
J\left(  s\right)  =%
{\displaystyle\int\nolimits_{\mathbb{Q}_{p}^{N}}}
\left\vert \mathcal{F}\Gamma\left(  s\right)  \left(  \xi\right)  \right\vert
^{2}d\mu\left(  \xi\right)  .
\]
We now consider the pathwise integral%
\[
\mathcal{J}^{0}\left(  t,x\right)  =%
{\displaystyle\int\nolimits_{0}^{t}}
{\displaystyle\int\nolimits_{\mathbb{Q}_{p}^{N}}}
b\left(  u^{0}\left(  t-s,x-y\right)  \right)  \Gamma\left(  s,y\right)
d^{N}yds.
\]
By applying the Cauchy-Schwartz inequality with respect to the finite measure
$\Gamma\left(  s,y\right)  d^{N}yds$ on $\left[  0,T\right]  \times
\mathbb{Q}_{p}^{N}$ and by using the Lipschitz property of $b$, one gets%
\begin{equation}
\left\vert \mathcal{J}^{0}\left(  t,x\right)  \right\vert ^{2}\leq C%
{\displaystyle\int\nolimits_{0}^{t}}
{\displaystyle\int\nolimits_{\mathbb{Q}_{p}^{N}}}
\left(  1+\left\vert u^{0}\left(  t-s,x-y\right)  \right\vert ^{2}\right)
\Gamma\left(  s,y\right)  d^{N}yds, \label{EqE}%
\end{equation}
which is uniformly bounded with respect to $t$ and $x$. This fact together
with (\ref{EqD}) imply that $\left\{  u^{1}\left(  t,x\right)  ,\left(
t,x\right)  \in\left[  0,T\right]  \times\mathbb{Q}_{p}^{N}\right\}  $ is a
well-defined measurable process, cf. Proposition \ref{Prop2}. In addition , by
(\ref{EqD})-(\ref{EqE}) and Hypothesis A,
\[
\sup_{\left(  t,x\right)  \in\left[  0,T\right]  \times\mathbb{Q}_{p}^{N}%
}E\left(  \left\vert u^{1}\left(  t,x\right)  \right\vert ^{2}\right)
<+\infty.
\]
Consider now the case $n>1$ and assume that $\left\{  u^{n}\left(  t,x\right)
,\left(  t,x\right)  \in\left[  0,T\right]  \times\mathbb{Q}_{p}^{N}\right\}
$ is a well-defined measurable process satisfying (\ref{EqC}). By the same
arguments as above, one proves that
\begin{equation}
E\left(  \left\vert \mathcal{I}^{n+1}\left(  t,x\right)  \right\vert
^{2}\right)  \leq C%
{\displaystyle\int\nolimits_{0}^{t}}
\sup_{z\in\mathbb{Q}_{p}^{N}}E\left(  1+\left\vert u^{n}\left(  s,z\right)
\right\vert ^{2}\right)
{\displaystyle\int\nolimits_{\mathbb{Q}_{p}^{N}}}
\left\vert \mathcal{F}\Gamma\left(  t-s\right)  \left(  \xi\right)
\right\vert ^{2}d\mu\left(  \xi\right)  ds, \label{EI_n+1}%
\end{equation}
and that%
\begin{equation}
E\left(  \left\vert \mathcal{J}^{n+1}\left(  t,x\right)  \right\vert
^{2}\right)  \leq C%
{\displaystyle\int\nolimits_{0}^{t}}
\sup_{y\in\mathbb{Q}_{p}^{N}}E\left(  1+\left\vert u^{n}\left(
t-s,x-y\right)  \right\vert ^{2}\right)
{\displaystyle\int\nolimits_{\mathbb{Q}_{p}^{N}}}
\Gamma\left(  s,y\right)  d^{N}yds. \label{EJ_n+1}%
\end{equation}
Hence the integrals $\mathcal{I}^{n+1}\left(  t,x\right)  $ and $\mathcal{J}%
^{n+1}\left(  t,x\right)  $ exist, so that $u^{n+1}$ is a well-defined
measurable process satisfying (\ref{EqC}).

\textbf{§2. }We now show that%
\begin{equation}
\sup_{n\geq0}\sup_{\left(  t,x\right)  \in\left[  0,T\right]  \times
\mathbb{Q}_{p}^{N}}E\left(  \left\vert u^{n}\left(  t,x\right)  \right\vert
^{2}\right)  <+\infty. \label{EqF}%
\end{equation}
Indeed, by using the estimates (\ref{EI_n+1})-(\ref{EJ_n+1}), we have%
\[
E\left(  \left\vert u^{n+1}\left(  t,x\right)  \right\vert ^{2}\right)  \leq
C\left(  1+%
{\displaystyle\int\nolimits_{0}^{t}}
\left(  1+\sup_{z\in\mathbb{Q}_{p}^{N}}E\left(  \left\vert u^{n}\left(
s,z\right)  \right\vert ^{2}\right)  \right)  \left(  J\left(  t-s\right)
+1\right)  \right)  ds.
\]
Now (\ref{EqF}) follows from the version of Gronwall's Lemma presented in
\cite[Lemma 15]{Dalang}.

\textbf{§3. }$u^{n}\left(  t,x\right)  $ $\underrightarrow{L^{2}\left(
\Omega\right)  }$\ \ $u\left(  t,x\right)  $\ uniformly in $x\in\mathbb{Q}%
_{p}^{N}$, $t\in\left[  0,T\right]  $.

Following the same ideas as in the proof of \cite[Theorem 13]{Dalang}, we take%
\[
M_{n}\left(  t\right)  :=\sup_{\left(  s,x\right)  \in\left[  0,t\right]
\times\mathbb{Q}_{p}^{N}}E\left(  \left\{  u^{n+1}\left(  s,x\right)
-u^{n}\left(  s,x\right)  \right\}  ^{2}\right)  .
\]
By using Proposition \ref{Prop2}, the Lipschitz property of $b$ and $\sigma$,
and by applying the same arguments as above, one gets%
\[
M_{n}\left(  t\right)  \leq C%
{\displaystyle\int\nolimits_{0}^{t}}
M_{n-1}\left(  s\right)  \left(  J\left(  t-s\right)  +1\right)  ds.
\]
Now by applying Gronwall's Lemma presented in \cite[Lemma 15]{Dalang}, we get%
\[
\lim_{n\rightarrow+\infty}\left(  \sup_{\left(  s,x\right)  \in\left[
0,t\right]  \times\mathbb{Q}_{p}^{N}}E\left(  \left\vert u^{n+1}\left(
s,x\right)  -u^{n}\left(  s,x\right)  \right\vert ^{2}\right)  \right)  =0.
\]
Hence $\left\{  u^{n}\left(  t,x\right)  \right\}  _{n\in\mathbb{N}}$
converges uniformly in $L^{2}\left(  \Omega\right)  $ to a limit $u\left(
t,x\right)  $. From this fact, we get
\begin{equation}
\lim_{n\rightarrow+\infty}\sup_{\left(  s,x\right)  \in\left[  0,T\right]
\times\mathbb{Q}_{p}^{N}}E\left(  \left\vert u^{n}\left(  t,x\right)
-u\left(  t,x\right)  \right\vert ^{2}\right)  =0\text{.} \label{EqG1}%
\end{equation}
Finally, by (\ref{EqG1})-(\ref{EqF}),%
\[
E\left(  \left\vert u\left(  t,x\right)  \right\vert ^{2}\right)
=\lim_{n\rightarrow+\infty}E\left(  \left\vert u^{n}\left(  t,x\right)
\right\vert ^{2}\right)  \leq\sup_{n\geq0}\sup_{\left(  t,x\right)  \in\left[
0,T\right]  \times\mathbb{Q}_{p}^{N}}E\left(  \left\vert u^{n}\left(
t,x\right)  \right\vert ^{2}\right)  <+\infty.
\]

\textbf{§4. }The process $\left\{  u\left(  t,x\right)  ,\left(  t,x\right)
\in\left[  0,T\right]  \times\mathbb{Q}_{p}^{N}\right\}  $ is $L^{2}\left(
\Omega\right)  $-continuous and has a jointly measurable version.

The proof of this fact is based on the following result. Let $\mathcal{L}$ be
a complete separable metric space, and $\mathcal{B}(\mathcal{L})$ the $\sigma
$-algebra of Borel sets of $\mathcal{L}$, and let $X_{s}$, $s\in\mathcal{L}$
be a real stochastic process on $\left(  \Omega,\mathcal{F},P\right)  $, where
real means $\left[  -\infty,+\infty\right]  $-valued. The process $X_{s}$,
$s\in\mathcal{L}$, is jointly measurable if the map $\left(  s,\omega\right)
\rightarrow X_{s}\left(  \omega\right)  $ is $\mathcal{B}(\mathcal{L}%
)\times\mathcal{F}$-measurable. Let $\mathcal{M}$ be the space of all real
random variables on $\left(  \Omega,\mathcal{F},P\right)  $ with the topology
of convergence in probability. Then $X_{s}$, $s\in\mathcal{L}$, has a jointly
measurable modification if and only if the map from $\mathcal{L}$\ to
$\mathcal{M}$ taking $s$ to $\left[  X_{s}\right]  $, the class of $X_{s}$ in
$\mathcal{M}$, is measurable, see (\cite{Chung-Doob})-(\cite[Theorem 3]{Cohn}).

In our case, $\mathcal{L}=\left(  \left[  0,T\right]  \times\mathbb{Q}_{p}%
^{N},d\right)  $ with
\begin{equation}
d\left(  \left(  t,x\right)  ,\left(  t^{\prime},x^{\prime}\right)  \right)
:=\max\left\{  \left\vert t-t^{\prime}\right\vert ,\left\Vert x-x^{\prime
}\right\Vert _{p}\right\}  . \label{metric}%
\end{equation}
Then $\mathcal{B}(\left[  0,T\right]  \times\mathbb{Q}_{p}^{N})=\mathcal{B}%
(\left[  0,T\right]  )\times\mathcal{B}(\mathbb{Q}_{p}^{N})$. It is sufficient
to show that the map from $\left[  0,T\right]  \times\mathbb{Q}_{p}^{N}$ to
$\mathcal{M}$ taking $\left(  t,x\right)  $ to $u\left(  t,x\right)  $ is
continuos in $L^{2}\left(  \Omega\right)  $. And since the convergence of
$u^{n+1}\left(  t,x\right)  $ to $u\left(  t,x\right)  $ is uniform in
$L^{2}\left(  \Omega\right)  $, it is sufficient to show that $u^{n+1}\left(
t,x\right)  $ is $L^{2}$-continuous. In order to do this, we have to verify
that
\begin{equation}
\lim_{h\rightarrow0}E\left(  \left\vert u^{n+1}\left(  t,x\right)
-u^{n+1}\left(  t+h,x\right)  \right\vert ^{2}\right)  =0\text{ \ }
\label{conver1}%
\end{equation}
and%
\begin{equation}
\lim_{x\rightarrow y}E\left(  \left\vert u^{n+1}\left(  t,x\right)
-u^{n+1}\left(  t,y\right)  \right\vert ^{2}\right)  =0. \label{conver2}%
\end{equation}
Indeed, (\ref{conver1}) implies that $x\rightarrow u^{n+1}\left(  t,x\right)
$ is uniformly continuous in $L^{2}\left(  \Omega\right)  $ and (\ref{conver2}%
) implies that $t\rightarrow u^{n+1}\left(  t,x\right)  $ is continuous in
$L^{2}\left(  \Omega\right)  $, therefore $\left(  t,x\right)  \rightarrow
u^{n+1}\left(  t,x\right)  $ is continuous in $L^{2}\left(  \Omega\right)  $.
The proof of this fact follows from Hypotheses A and B by using the technique
given in \cite{Dalang} to prove Lemma 19.

\textbf{§5. }$u\left(  t,x\right)  $ is a solution of (\ref{SolutionIVP}).

We set%
\[
\mathcal{I}\left(  t,x\right)  :=%
{\displaystyle\int\nolimits_{0}^{t}}
{\displaystyle\int\nolimits_{\mathbb{Q}_{p}^{N}}}
\Gamma\left(  t-s,x-y\right)  \sigma\left(  u\left(  s,y\right)  \right)
W\left(  ds,d^{N}y\right)
\]
and%
\[
\mathcal{J}\left(  t,x\right)  :=%
{\displaystyle\int\nolimits_{0}^{t}}
{\displaystyle\int\nolimits_{\mathbb{Q}_{p}^{N}}}
b\left(  u\left(  t-s,x-y\right)  \right)  \Gamma\left(  s\right)  d^{N}yds.
\]
In order to to establish Step 5, it is sufficient to show that
\begin{equation}
\lim_{n\rightarrow+\infty}\sup_{\left(  t,x\right)  \in\left[  0,T\right]
\times\mathbb{Q}_{p}^{N}}E\left(  \left\vert \mathcal{I}^{n}\left(
t,x\right)  -\mathcal{I}\left(  t,x\right)  \right\vert ^{2}\right)  =0,
\label{EqH}%
\end{equation}
and that
\begin{equation}
\lim_{n\rightarrow+\infty}\sup_{\left(  t,x\right)  \in\left[  0,T\right]
\times\mathbb{Q}_{p}^{N}}E\left(  \left\vert \mathcal{J}^{n}\left(
t,x\right)  -\mathcal{J}\left(  t,x\right)  \right\vert ^{2}\right)  =0.
\label{EqI}%
\end{equation}
To show (\ref{EqH}) we proceed as follows. By The Lipschitz property of
$\sigma$, Proposition \ref{Prop2} and Hypothesis A,%
\begin{align*}
&  E\left(  \left\vert \mathcal{I}^{n}\left(  t,x\right)  -\mathcal{I}\left(
t,x\right)  \right\vert ^{2}\right) \\
&  \leq E\left(  \left\vert
{\displaystyle\int\nolimits_{0}^{t}}
{\displaystyle\int\nolimits_{\mathbb{Q}_{p}^{N}}}
\Gamma\left(  t-s,x-y\right)  \left[  \sigma\left(  u^{n-1}\left(  s,y\right)
\right)  -\sigma\left(  u\left(  s,y\right)  \right)  \right]  W\left(
ds,d^{N}y\right)  \right\vert ^{2}\right) \\
&  \leq C%
{\displaystyle\int\nolimits_{0}^{t}}
\sup_{z\in\mathbb{Q}_{p}^{N}}\left(  \left\vert u^{n-1}\left(  s,z\right)
-u\left(  s,z\right)  \right\vert ^{2}\right)
{\displaystyle\int\nolimits_{\mathbb{Q}_{p}^{N}}}
\left\vert \mathcal{F}\Gamma\left(  t-s\right)  \left(  \xi\right)
\right\vert ^{2}d\mu\left(  \xi\right)  ds\\
&  \leq C\sup_{z\in\mathbb{Q}_{p}^{N}}\left(  \left\vert u^{n-1}\left(
s,z\right)  -u\left(  s,z\right)  \right\vert ^{2}\right)  ,
\end{align*}
this last term tends to zero as $n$ tend to infinity. The case (\ref{EqI}) can
be treated in a similar form. By the results of Paragraph 4, the process
$\left\{  u\left(  t,x\right)  ,\left(  t,x\right)  \in\left[  0,T\right]
\times\mathbb{Q}_{p}^{N}\right\}  $ has a measurable version that satisfies
(\ref{SolutionIVP}).

\textbf{§6. }$u\left(  t,x\right)  $ is the unique solution of
(\ref{SolutionIVP}) satisfying (\ref{EqExpe}).

This fact can be checked by using standards arguments.
\end{proof}

\begin{acknowledgement}
The author wishes to thank to Anatoly Kochubei and Sergii Torba for fruitful
discussions and remarks.\bigskip
\end{acknowledgement}

\end{document}